\documentclass[leqno,onefignum,onetabnum]{siamltex1213}
\usepackage{amsfonts,amsmath,amssymb}
\usepackage{booktabs,ctable,threeparttable,multirow}
\usepackage{graphicx,boxedminipage,appendix}
\usepackage{bm,enumerate,algorithm,algorithmic}
\usepackage[caption=false]{subfig}
\usepackage{mathrsfs}

\usepackage{amsthm}
\usepackage{lineno}
\usepackage{enumerate,enumitem}
\usepackage{soul}
\usepackage{hyperref}
\usepackage{color,xcolor}

\usepackage{mathptmx}      
\RequirePackage{fix-cm}
\usepackage{latexsym}

\allowdisplaybreaks[4]

\newtheorem{theoremmy}{Theorem}[section]
\newtheorem{lemmamy}[theoremmy]{Lemma}

\newtheorem{exper}[theoremmy]{Experiment}
\newtheorem{remark}{Remark}
\numberwithin{equation}{section}

\newcommand{\mi}{\mathrm{i}}

\newcommand{\spans}{\mathrm{span}}
\newcommand{\gap}{\mathrm{gap}}

\newcommand{\sign}{\mathrm{sign}}

\newcommand{\TT}{\mathrm{T}}
\newcommand{\HH}{\mathrm{H}}
\newcommand{\FF}{\mathrm{F}}

\newcommand{\PP}{\mathcal{P}}

\newcommand{\OO}{\mathcal{O}}
\newcommand{\UU}{\mathcal{U}}
\newcommand{\VV}{\mathcal{V}}
\newcommand{\WW}{\mathcal{W}}
\newcommand{\XX}{\mathcal{X}}
\newcommand{\YY}{\mathcal{Y}}
\newcommand{\ZZ}{\mathcal{Z}}

\title{A generalized skew-symmetric Lanczos bidiagonalization method 
	for computing several extreme eigenpairs of a large 
	skew-symmetric/symmetric positive definite matrix pair
	\thanks{This work was supported by the Youth Fund of the National 
	Science Foundation of China (No. 12301485) and the Youth Program 
	of the Natural Science Foundation of Jiangsu Province (No. BK20220482).}}

\author{ Jinzhi Huang\thanks{School of Mathematical Sciences, Soochow 
		University, 215006 Suzhou, China (\url{jzhuang21@suda.edu.cn}).}
}

\begin{document}
\maketitle

\begin{abstract}
A generalized skew-symmetric Lanczos bidiagonalization (GSSLBD) 
method is proposed to compute several extreme eigenpairs 
of a large matrix pair $(A,B)$, where $A$ is skew-symmetric 
and $B$ is symmetric positive definite. 
The underlying GSSLBD process produces two sets of $B$-orthonormal 
generalized Lanczos basis vectors that are also $B$-biorthogonal 
and a series of bidiagonal matrices whose singular values are 
taken as the approximations to the imaginary parts of the 
eigenvalues of $(A,B)$ and the corresponding left and right 
singular vectors premultiplied with the left and right generalized 
Lanczos basis matrices form the real and imaginary parts of the 
associated approximate eigenvectors.  
A rigorous convergence analysis is made on the desired eigenspaces 
approaching the Krylov subspaces generated by the GSSLBD process 
and accuracy estimates are made for the approximate eigenpairs. 
In finite precision arithmetic, it is shown that the 
semi-$B$-orthogonality and semi-$B$-biorthogonality 
of the computed left and right generalized Lanczos vectors 
suffice to compute the eigenvalues accurately. 
An efficient partial reorthogonalization strategy is adapted 
to GSSLBD in order to maintain the desired semi-$B$-orthogonality and 
semi-$B$-biorthogonality. 
To be practical, an implicitly restarted GSSLBD algorithm, abbreviated 
as IRGSSLBD, is developed with partial $B$-reorthogonalizations.
Numerical experiments illustrate the robustness and overall 
efficiency of the IRGSSLBD algorithm.
\end{abstract}

\begin{keywords}
skew-symmetric matrix, symmetric positive definite matrix,
generalized eigenproblem, eigenvalue, eigenvector, 
singular value decomposition, singular value, singular vector, 
generalized skew-symmetric Lanczos bidiagonalization,
partial reorthogonalization, implicit restart
\end{keywords}

\begin{AMS}
65F15, 15A18, 65F10, 65F25
\end{AMS}

\pagestyle{myheadings}
\thispagestyle{plain}
\markboth{A  GENERALIED SKEW-SYMMETRIC LANCZOS BIDIAGONALIZATION METHOD}
{JINZHI HUANG}

\section{Introduction}\label{sec:1}
Let $A\in\mathbb{R}^{n\times n}$ and $B\in\mathbb{R}^{n\times n}$ 
be large scale and possibly sparse matrices with $A$ skew-symmetric 
and $B$ symmetric positive definite, i.e., $A^\TT=-A$, $B=B^\TT$ and 
$B\succ0$, where the superscript $\TT $ denotes the transpose of a 
matrix or vector.
We consider the generalized eigenvalue problem of $(A,B)$: 
\begin{equation}\label{eigA}
  Ax=\lambda Bx,
\end{equation}
where $\lambda\in\mathbb{C}$ is an eigenvalue of $(A,B)$ and 
$x\in\mathbb{C}^{n}$ with  $\|x\|_{B}=1$ is the 
corresponding eigenvector.  
Throughout this paper, $\|\cdot\|_B$ denotes the $B$-norm of 
a matrix or vector. 

The real skew-symmetric eigenproblems and skew-symmetric/symmetric 
generalized eigenproblems arise in various applications 
\cite{apel2002structured,del2005computation,mehrmann1991autonomous,
	mehrmann2012implicitly,penke2020high,wimmer2012algorithm,
	yan1999approximate,zhong1995direct}. 
For a small to medium sized skew-symmetric matrix $A$, several numerical 
algorithms have been developed to compute its eigendecompositions
in real arithmetic \cite{fernando1998accurately,paardekooper1971eigenvalue,
	penke2020high,ward1978eigensystem}. 
For a large scale skew-symmetric $A$, Huang and Jia \cite{huang2024skew}
make full use of the skew-symmetric structure and develop a skew-symmetric 
Lanczos bidiagonalization (SSLBD) process that successively generates two sets
of orthonormal and biorthogonal Lanczos vectors and a sequence of 
bidiagonal projection matrices of $A$ with respect to the 
Lanczos subspaces, whose extreme, i.e., largest and smallest, 
singular values are reasonable approximations to the absolute 
values of the (pure imaginary) extreme eigenvalues of $A$. 
Based on this process, they propose an implicit restart SSLBD algorithm 
for computing several extreme conjugate eigenpairs of a large skew-symmetric 
matrix in pure real arithmetic. 
For a large scale skew-symmetric/symmetric matrix pair $(A,B)$, 
Mehrmann, Schr{\"o}der and Simoncini \cite{mehrmann2012implicitly} 
present an implicit restart Krylov subspace method to compute its
several eigenpairs with the eigenvalues closest to a given target $\tau$. 
In this method, the concerned interior generalized eigenproblem 
of $(A,B)$ is mathematically transformed as the extreme eigenproblem 
of the non-specially structured matrix $K=(B+\tau A)^{-1}A(B-\tau A)^{-1}A$, 
and the Arnoldi process \cite{stewart2001matrix} is performed on 
$K$ to compute the relevant eigenpairs, 
from which the desired eigenpairs of $(A,B)$ are recovered. 
 
However, when it comes to the eigenproblem of a large scale 
skew-symmetric/symmetric positive definite matrix pair $(A,B)$, 
there still lacks a high-performance numerical algorithm that 
can fully exploit the special structure of the matrix pair 
and compute several eigenvalues and/or the associated eigenvectors 
only in real arithmetic. 
In this paper, we fill this chink by generalizing the work of 
\cite{huang2024skew} and proposing a generalized SSLBD (GSSLBD) 
method to compute the extreme eigenpairs of $(A,B)$ 
conjugate pairwise in pure real arithmetic. 

We first show that for $A$ skew-symmetric and $B$ symmetric 
positive definite, the generalized eigenproblem of $(A,B)$ 
is equivalent to the eigenproblem of a certain skew-symmetric matrix $H$. 
Based on this, we generalize the SSLBD process of 
\cite{huang2024skew} to the GSSLBD process, performing which on 
$(A,B)$ delivers two sets of $B$-orthonormal and $B$-biorthogonal 
vectors, called the left and right generalized Lanczos vectors, 
and a series of bidiagonal projection matrix of $A$ onto the Krylov 
subspaces spanned by these two sets of vectors, called the left 
and right generalized Lanczos subspaces, with respect to $B$-norm. 
Then performing the Rayleigh-Ritz projection approach, i.e., the 
standard extraction method, on the eigenvalue problem of $(A,B)$ 
with respect to the generalized Lanczos subspaces, we compute 
approximations to the absolute values of the pure imaginary 
eigenvalues of $(A,B)$ and/or approximations to the real and 
imaginary parts of the associated eigenvectors.
Once convergence occurs, we recover the desired approximate eigenpairs 
of $(A,B)$ conjugate pairwise directly from these approximations. 
The resulting method is called the GSSLBD method. 
We estimate the distance between the two-dimensional eigenspace 
associated with a conjugate eigenvalue pair of $(A,B)$ and 
the underlying Krylov subspaces with respect to $B$-norm. 
Then in term of these distance, we derive a priori error bounds for 
the computed approximate eigenvalues and the corresponding 
approximate eigenspaces. 
These results illustrate that the GSSLBD method suits well for the 
computation of extreme eigenpairs of $(A,B)$.

In finite precision arithmetic, just as the Lanczos vectors generated 
by the SSLBD process loss orthogonality and biorthogonality 
gradually, the left and right generalized Lanczos vectors produced 
by the GSSLBD process loss the $B$-orthogonality and 
$B$-biorthogonality progressively, causing the ghost phenomena of 
approximate eigenvalues, i.e., spurious approximation copies of some 
eigenvalues appear frequently. 
For the correctness and reliability of the GSSLBD method, 
$B$-reorthogonalization is absolutely necessary. 
Fortunately, we prove that the
semi-$B$-orthogonality and semi-$B$-biorthogonality of the two 
sets of generalized Lanczos vectors suffice to compute the 
eigenvalues of $(A,B)$ accurately. 
In other words, the left and right generalized Lanczos vectors  
only need to be $B$-orthonormal and $B$-biorthogonal to the 
level of $\mathcal{O}(\sqrt{\epsilon})$ with $\epsilon$ the 
machine precision. 
To this end, we generalize the partial reorthogonalization 
strategies proposed in \cite{huang2024skew,larsen1998lanczos} 
for the LBD and SSLBD processes,  
and design a partial $B$-reorthogonalization module for the GSSLBD 
process so as to maintain the desired numerical semi-$B$-orthogonality 
and semi-$B$-biorthogonality of the generalized 
Lanczos vectors. 
With the partial $B$-reorthogonalization, the GSSLBD method 
works as correctly as but more efficiently than it does with 
full $B$-reorthogonalization where all the generalized Lanczos 
vectors are $B$-reorthogonalized so that they are $B$-orthonormal 
and $B$-biorthogonal to the working precision, 
as confirmed by the numerical experiments. 
   
As the subspaces expand, the GSSLBD method becomes prohibitive 
due to the exorbitant computational cost and storage requirement. 
To be practical, it is absolutely vital to restart the method appropriately.  
Implicit restart was originally introduced by Sorensen 
\cite{sorensen1992implicit} 
for large eigenproblems and then evolved by Larsen \cite{larsen2001combining} 
and Jia and Niu \cite{jia2003implicitly,jia2010refined} for large 
singular value decomposition (SVD) computations. 
Recently, Huang and Jia \cite{huang2024skew} have adapted it to 
the computation of partial spectral decomposition of large real  
skew-symmetric matrices. 
In this paper, we customize the implicit restart technique for the 
GSSLBD method and restart the method with increasingly better initial 
subspaces when the maximum dimension of the subspaces allowed is attained.    
As a result, we propose an implicit GSSLBD algorithm with 
partial $B$-reorthogonalization for computing several extreme 
conjugate eigenpairs of a real large skew-symmetric/symmetric 
positive definite matrix pair $(A,B)$ in pure real arithmetic. 
We implement our implicit restart GSSLBD algorithm and the Matlab 
built-in function {\sf eigs} on numerous problems to compare their 
numerical performance on computing several largest conjugate 
eigenvalues in magnitude of $(A,B)$ and the corresponding 
eigenvectors, showing that our algorithm is competitive with 
and generally more robust and efficient than {\sf eigs}. 
In addition, we numerically illustrate that the implicit restart 
GSSLBD algorithm also performs well for the computation of several 
smallest conjugate eigenpairs in magnitude of $(A,B)$.
 
The rest of this paper is organized as follows. 
In section~\ref{sec2} we introduce the GSSLBD process and GSSLBD method. 
In section~\ref{sec3} we present the theoretical analysis 
on the GSSLBD method. 
Section~\ref{sec:4} is devoted to the partial $B$-reorthogonalization 
scheme and we propose an implicit restart GSSLBD algorithm. 
We report the numerical experiments in section~\ref{sec:5} and 
conclude the paper in section~\ref{sec:6}. 
 
Throughout this paper, we denote by $\mi$ the imaginary unit,
by $X^\HH$, $\|X\|$ and $\|X\|_{\FF}$ the conjugate transpose, 
$2$-norm and Frobenius norms of $X$, respectively. 
Denote by $I_{k}$ and $\bm{0}_{k,\ell}$ the identity and zero 
matrices of order $k$ and $k\times \ell$, respectively,  
the subscripts of which are omitted whenever they are 
clear from the context. 

\section{The GSSLBD method}\label{sec2}
Since $B$ is symmetric and positive definite, it has the 
following square decomposition:
\begin{equation}\label{squaredec}
	B=M^2,	
\end{equation}
where $M=B^{\frac{1}{2}}\in\mathbb{R}^{n\times n}$ is also 
symmetric and positive definite. 
The following theorem reveals the close relationship between the 
spectral decompositions of $(A,B)$ and a certain skew-symmetric matrix, 
the proof to which is straightforward and therefore omitted. 
\begin{theoremmy}\label{thm1}
	For an eigenpair $(\lambda,x)$ of $(A,B)$ where $A$ is skew-symmetric 
	and $B$ is symmetric positive definite, $(\lambda,y:=Mx)$ is 
	an eigenpair of the skew-symmetric matrix 
	\begin{equation}\label{Achek}
		H=M^{-1}AM^{-1}.
	\end{equation}
\end{theoremmy}

Theorem~\ref{thm1} shows that the computation of the eigenpairs of $(A,B)$ 
is equivalent to that of the skew-symmetric $H$. 
Specifically, Theorem~2.1 of \cite{huang2024skew} shows that the spectral 
decomposition of the skew-symmetric $H=Y\Lambda Y^\HH$ 
has the following form: 
\begin{equation}\label{spectralhatA}
Y=\begin{bmatrix}
\frac{1}{\sqrt2}(W+\mi Z)&\frac{1}{\sqrt2}(W-\mi Z)&Y_0
\end{bmatrix}
\qquad\mbox{and}\qquad
\Lambda=\diag\{\mi\Sigma,-\mi\Sigma,\bm{0}_{\ell_0}\},
\end{equation}
where $W\in\mathbb{R}^{n\times \ell}$, $Z\in\mathbb{R}^{n\times \ell}$ 
and $Y_0\in\mathbb{R}^{n\times\ell_0}$ make the matrix $[W,Z,Y_0]$ orthogonal, 
$\Sigma=\diag\{\sigma_1,  \dots,\sigma_{\ell}\}$ with 
$\sigma_1,\dots,\sigma_{\ell}>0$, 
integer $\ell$ is the total number of nonzero conjugate eigenvalue 
pairs that $H$ owns, and $\ell_0=n-2\ell$ equals to the dimension 
of the null space $\mathcal{N}(H)$ of $H$ or, equivalently, that 
of $\mathcal{N}(A)$.
Therefore, the spectral decomposition of $(A,B)$ is
\begin{equation}\label{spectralAB}
AX=BX\Lambda
\qquad\mbox{with}\qquad
X=\begin{bmatrix} 
	\frac{1}{\sqrt2}(U+\mi V)&\frac{1}{\sqrt2}(U-\mi V)&X_0
\end{bmatrix},
\end{equation}
where $\Lambda$ is defined in \eqref{spectralhatA}, $X=M^{-1}Y$ and 
\begin{equation} \label{relations}
	U=M^{-1}W, \qquad
	V=M^{-1}Z, \qquad
	X_{0}=M^{-1}Y_0
\end{equation}
make the matrix 
$[U,V,X_0]$ $B$-orthogonal, i.e., 
\begin{equation}\label{Borthos}
U^\TT BU=I,\quad
V^\TT BV=I,\quad
U^\TT BV=\bm{0} ,\quad X_0^\TT BX_0=I,\quad
U^\TT BX_0=\bm{0},\quad
V^\TT BX_0=\bm{0}.
\end{equation}

Denote by $u_j,v_j$ and $w_j,z_j$ the $j$th columns of $U,V$ and $W,Z$, respectively. 
Then \eqref{relations} implies that 
\begin{equation}\label{relations2}
	w_j=Mu_j\qquad\mbox{and}\qquad
	z_j=Mv_j.
\end{equation} 
On the basis of the above analyses, throughout this paper, we denote 
the nontrivial eigenpairs of $(A,B)$, i.e., those corresponding to 
nonzero eigenvalues, by 
\begin{equation}\label{eigenpairs}
(\lambda_{\pm j},x_{\pm j})=\left(\pm\mi\sigma_j,
\frac{1}{\sqrt2}(u_j\pm\mi v_j)\right),\qquad j=1,\dots,\ell.	
\end{equation}
Combining the spectral decomposition \eqref{spectralhatA} of $H$ and 
Theorem~2.1 of \cite{huang2024skew}, we know that 
$(\sigma_j,w_j,\pm z_j)$, $j=1,\dots,\ell$ are singular triplets of $H$. 
Therefore, in order to obtain the conjugate eigenpairs 
$(\lambda_{\pm j},x_{\pm j})$ of $(A,B)$, one can first compute the 
singular triplet $(\sigma_j,w_j,z_j)$ or $(\sigma_j,w_j,-z_j)$ of 
$H$ using some skew-symmetric SVD method, e.g., the SSLBD method 
in \cite{huang2024skew}, 
and then recover the approximate eigenpairs of $(A,B)$ 
from the converged approximate singular triplets of $H$ according 
to the above relationships. 
Based on these, we call $(\sigma_j,u_j,v_j)$ a generalized spectral 
triplet of $(A,B)$ with $\sigma$ the generalized spectral value and 
$u_j$, $v_j$ the left and right generalized spectral vectors, respectively.

For ease of presentation, throughout this paper, we assume that 
all the eigenvalues $\lambda_{\pm j}$ of $(A,B)$ in \eqref{eigenpairs} 
are simple, and the generalized spectral values 
$\sigma_1,\dots,\sigma_\ell$ are labeled in decreasing order:
\begin{equation}\label{labeleig}
  \sigma_1>\sigma_2>\dots>\sigma_\ell>0.
\end{equation}
We aim at computing the $k$ pairs of extreme, e.g., largest, conjugate 
eigenvalues $\lambda_{\pm 1},\dots,\lambda_{\pm k}$ of $(A,B)$
and/or the corresponding eigenvectors $x_{\pm 1},\dots,x_{\pm k}$.
By Theorem~\ref{thm1} and above, this is equivalent to 
computing the following partial SVD of the skew-symmetric $H$:
$$
(\Sigma_k,W_k,Z_k)= \left(\mathrm{diag}\{\sigma_1,\dots,\sigma_k\},
[w_1,\dots,w_k],[z_1,\dots,z_k]\right).
$$

\subsection{The GSSLBD process}\label{subsec:1}
Algorithm~\ref{alg1} sketches the $m$-step SSLBD process on the
skew-symmetric $H$ that is presented in \cite{huang2024skew}. 
Given an initial normalized vector $\hat q_1$, it successively
computes the orthonormal and biorthogonal left and right Lanczos 
vectors $\hat p_1,\dots,\hat p_m$ and $\hat q_1,\dots,\hat q_{m+1}$ 
that form orthonormal bases of the (if not break) $m$-dimensional 
Krylov subspaces  
\begin{equation}\label{WWZZmm}
	\WW_m=\mathcal{K}_m(H^2,\hat p_1)  \qquad\mbox{and}\qquad	
	\ZZ_m=\mathcal{K}_m(H^2,\hat q_1) ,
\end{equation}
of $H^2$ with respect to the normalized initial vectors 
$\hat p_1=\frac{H\hat q_1}{\|H\hat q_1\|}$ and $\hat q_1$, respectively.  

\begin{algorithm}[htbp]
\caption{The $m$-step SSLBD process.}
\begin{algorithmic}[1]\label{alg1}
\STATE{Initialization:\ Set $\beta_0=0$ and $\hat p_{0}=\bm{0}$,
and choose $\hat q_1\in\mathbb{R}^{n}$ with $\|\hat q_1\|=1$.}
\FOR{$j=1,\dots,m<\ell$}
\STATE{ Compute $\hat s_j=H\hat q_j-\beta_{j-1}\hat p_{j-1}$ and
	$\alpha_j=\|\hat s_j\|$.}  

\STATE{\textbf{if} $\alpha_j=0$ \textbf{then} break;
\textbf{else} calculate $\hat p_j=\frac{1}{\alpha_j}\hat s_j$.}   

\STATE{Compute $\hat t_j=-H\hat p_j-\alpha_j\hat q_j$ and
	$\beta_{j}=\|\hat t_j\|$.}  

\STATE{\textbf{if} $\beta_j=0$ \textbf{then} break;
\textbf{else} calculate $\hat q_{j+1}=\frac{1}{\beta_{j}}\hat t_j$.}  
\ENDFOR
\end{algorithmic}
\end{algorithm}

Denote
$\widehat P_j=[\hat p_1,\dots,\hat p_j]$ for $j=1,\dots,m$ and
$\widehat Q_j=[\hat q_1,\dots,\hat q_j]$ for $j=1,\dots,m+1$. 
Then they are orthonormal themselves and biorthogonal, 
i.e., $\widehat P_m^\TT \widehat Q_{m+1}=\bm{0}$.
Moreover, the $m$-step SSLBD process can be written in the matrix form:
\begin{equation}\label{LBDmat}
	\left\{\begin{aligned}
		&H\widehat Q_m=\widehat P_mG_m \\[0.5em]
		&H\widehat P_m=-\widehat Q_{m} G_{m}^\TT 
		-\beta_m\hat q_{m+1}e_{m,m}^\TT 
	\end{aligned}\right. 
\qquad\mbox{with}\qquad
G_m=\begin{bmatrix} 
	\alpha_1&\beta_1&&\\&\ddots&\ddots&\\&&\ddots&\beta_{m-1}\\&&&\alpha_{m}
\end{bmatrix},
\end{equation}
where $e_{m,m}$ is the $m$th column of $I_m$.  
 
Note that $H=M^{-1}AM^{-1}$ and $B=M^{2}$. 
Take $P_m=M^{-1}\widehat P_m$, $Q_{m}=M^{-1}\widehat Q_m$ and 
$q_{m+1}=M^{-1}\hat q_{m+1}$. 
Then it is obvious that the columns of $P_m$ and $Q_m$ are 
$B$-orthonormal and they are $B$-biorthogonal, i.e., 
\begin{equation}\label{Bortho}
	P_m^\TT BP_m=I,\qquad
	Q_{m+1}^\TT  BQ_{m+1}=I,	\qquad
	P_m^\TT B Q_{m+1}=\bm{0}.
\end{equation}
In addition, multiplying $M^{-1}$ both hand sides of the two 
relations in \eqref{LBDmat}, we obtain 
\begin{equation}\label{LBDmat2}
	\left\{\begin{aligned}
		&B^{-1}A Q_m= P_mG_m,\\[0.5em]
		&B^{-1} A P_m=-Q_{m} G_{m}^\TT -\beta_m  q_{m+1}e_{m,m}^\TT. 
	\end{aligned}\right. 
\end{equation} 
Modifying Algorithm~\ref{alg1}, we obtain the $m$-step generalized 
SSLBD (GSSLBD) process; see Algorithm~\ref{alg2}.

\begin{algorithm}[htbp]
	\caption{The $m$-step GSSLBD process.}
	\begin{algorithmic}[1]\label{alg2}
		\STATE{Initialization:\ Set $\beta_0=0$ and $  p_{0}=\bm{0}$,
			and choose $ q_1\in\mathbb{R}^{n}$ with $\|  q_1\|_B=1$.}
		\FOR{$j=1,\dots,m<\ell$}
		\STATE{Compute $ s_j= B^{-1}A q_j-\beta_{j-1} p_{j-1}$ and
			$\alpha_j=\| s_j\|_B$.} \label{skwbld1}
		
		\STATE{\textbf{if} $\alpha_j=0$ \textbf{then} break;
			\textbf{else} calculate $ p_j=\frac{1}{\alpha_j} s_j$.}  \label{skwbld2}
		
		\STATE{Compute $ t_j=-B^{-1}A  p_j-\alpha_j q_j$ and
			$\beta_{j}=\|t_j\|_B$.} \label{skwbld3}
		
		\STATE{\textbf{if} $\beta_j=0$ \textbf{then} break;
			\textbf{else} calculate $  q_{j+1}=\frac{1}{\beta_{j}} t_j$.} \label{skwbld4}
		\ENDFOR
	\end{algorithmic}
\end{algorithm}

Assume that Algorithm~\ref{alg2} does not break down for $m<\ell$. 
Then the $m$-step GSSLBD process computes the $B$-orthonormal 
and $B$-biorthogonal bases $\{p_j\}_{j=1}^m$ and 
$\{q_j\}_{j=1}^{m+1}$ of the Krylov subspaces 
\begin{equation}\label{defUV}
	\UU_m=\mathcal{K}_m((B^{-1}A)^2,p_1)
	\qquad\mbox{and}\qquad
	\VV_{m+1}=\mathcal{K}_{m+1}((B^{-1}A)^2,q_1)
\end{equation}
generated by $(B^{-1}A)^2$ and the $B$-norm unit length vectors 
$p_1=\frac{B^{-1}A q_1}{\|B^{-1}A q_1\|_{B}}$ and $q_1$, respectively. 
Take $\UU_m$ and $\VV_m$ as the left and right subspaces, from which 
approximations to the left and right generalized spectral vectors 
$u_j$'s and $v_j$'s of $(A,B)$ are extracted, respectively. 
Then by \eqref{Bortho} {\em any} approximate left and right generalized 
spectral vectors, i.e., the real and imaginary parts of an approximate 
eigenvector, extracted from $\UU_m$ and $\VV_m$ are $B$-biorthogonal,
a desired property as the exact left and right generalized spectral 
vectors of $(A,B)$ are so, too; see \eqref{Borthos}.

Once the GSSLBD process breaks down for some $m\leq \ell$, i.e., 
$\alpha_m=0$ or $\beta_m=0$, all the singular values of 
$G_{m-1}^{\prime}=[G_{m-1},\beta_{m-1}e_{m-1,m-1}]$ or $G_m$ 
are exact ones of $H$, as we shall see later, and $m-1$ or $m$ 
exact singular triplets of $H$ can be found, so that $m-1$ or 
$m$ exact conjugate eigenpairs of $(A,B)$ can be recovered.
Note that the singular values of $G_{m-1}^{\prime}$ and $G_m$ are simple
whenever the $m$-step GSSLBD process does not break down. 
Therefore, the GSSLBD process must break down no later than $\ell$ 
steps since $H$ has $\ell$ distinct nonzero singular values.  
 
\subsection{The GSSLBD method}\label{subsec:2}
With the left and right subspaces $\UU_m$ and $\VV_m$ 
defined by \eqref{defUV}, 
we implicitly take the Krylov subspaces in \eqref{WWZZmm}, 
which by definition are biorthogonal and satisfy 
\begin{equation}\label{WWZZm}
	\WW_m=M\UU_m \qquad\mbox{and}\qquad	
	\ZZ_m=M\VV_m,  
\end{equation}
as the left and right searching subspaces for the left and right singular 
vectors $w$ and $z$ of $H$, respectively.  
We adopt the standard Rayleigh--Ritz projection, i.e.,
the standard extraction approach \cite{saad2011numerical,stewart2001matrix},
to compute the approximate singular triplets 
$(\theta_j,\tilde w_j,\tilde z_j)$ 
of $H$ with the unit-length $\tilde w_j\in\WW_m$ 
and $\tilde z_j\in\ZZ_m$ such that 
\begin{equation}\label{standard}
	\left\{\begin{aligned}
		&H\tilde z_j-\theta_j\tilde w_j\perp\WW_m,\\
		&H\tilde w_j+\theta_j\tilde z_j\perp\ZZ_m,
	\end{aligned}\right.
	\qquad\qquad j=1,\dots,m.
\end{equation}
Denote $\tilde u_j=M^{-1}\tilde w_j$ and $\tilde v_j=M^{-1}\tilde z_j$ for 
$j=1,\dots,m$. Then $\tilde u_j\in\UU_m$ and $\tilde v_j\in\VV_m$ are of 
$B$-norm unit length and we take $(\theta_j,\tilde u_j,\tilde v_j)$ as 
approximations to the generalized spectral triplets of $(A,B)$, $j=1,\dots,m$. 
It is straightforward to verify that \eqref{standard} is equivalent to  
\begin{equation}\label{standard2}
	\left\{\begin{aligned}
		&B^{-1} A\tilde v_j-\theta_j\tilde u_j\perp_B \UU_m,\\
		&B^{-1}  A\tilde  u_j+\theta_j\tilde  v_j\perp_B \VV_m,
	\end{aligned}\right.
	\qquad\qquad j=1,\dots,m,
\end{equation}
where $\perp_B$ denotes the $B$-orthogonality, i.e., 
$P_m^{\TT}B(B^{-1} A\tilde v_j-\theta_j\tilde u_j)=\bm{0}$
and $Q_m^\TT B(B^{-1}  A\tilde  u_j+\theta_j\tilde  v_j)=\bm{0}$ 
as $P_m$ and $Q_m$ are the $B$-orthonormal basis 
matrices of $\UU_m$ and $\VV_m$, respectively. 
  
Set $\tilde u_j=P_mc_j$ and $\tilde v_j=Q_md_j$ for $j=1,\dots,m$. 
Making use of \eqref{LBDmat2}, relation \eqref{standard2} amounts 
to solving the following problem: 
\begin{equation}\label{Ritz}
	G_md_j=\theta_jc_j,\qquad
	G_m^\TT c_j=\theta_jd_j,\qquad
	j=1,\dots,m;
\end{equation}
that is, $(\theta_j,c_j,d_j)$, $j=1,\dots,m$ are the singular
triplets of $G_m$.
Therefore, the standard extraction approach computes the SVD
of $G_m$ with the singular values ordered decreasingly: 
$\theta_1>\theta_2>\dots>\theta_m$, and takes
\begin{equation}\label{appeigenpair}
(\theta_j,\tilde u_j=P_mc_j,\tilde v_j=Q_md_j),\qquad j=1,\dots,m	
\end{equation}
as approximations to some generalized spectral triplets of $(A,B)$.
The resulting method is called the GSSLBD method. 
We call $(\theta_j,\tilde u_j, \tilde v_j)$ the generalized 
Ritz approximation of $(A,B)$ with respect to the left and right 
subspaces $\UU_m$ and $\VV_m$; particularly, we call $\theta_j$ 
the Ritz values, and call $\tilde u_j$ and $\tilde v_j$ the corresponding 
left and right generalized Ritz vectors. 

Since $\widehat Q_mH^\TT H\widehat Q_m=G_m^\TT G_m$ by \eqref{LBDmat}, 
the Cauchy interlace theorem of eigenvalues 
(cf. \cite[Theorem 10.1.1, pp.203]{parlett1998symmetric}) indicates that 
for $j$ fixed, as $m$ increases, the singular values $\theta_j$ and 
$\theta_{m-j+1}$ of $G_m$ monotonically converge to the singular 
values $\sigma_j$ and $\sigma_{\ell-j+1}$ of $H_m$ from below and 
above, respectively.
Therefore, we can take $\{(\theta_j,\tilde u_j,\tilde v_j)\}_{j=1}^k$ 
as the approximations to the $k$ largest generalized spectral triplets
$\{(\sigma_j,u_j,v_j)\}_{j=1}^k$ of $(A,B)$ for $k\ll m$.
Likewise, we might use $\{(\theta_{m-j+1},\tilde u_{m-j+1},
\tilde v_{m-j+1})\}_{j=1}^k$ to approximate the $k$ smallest
generalized spectral triplets 
$\{(\sigma_{\ell-j+1},u_{\ell-j+1},v_{\ell-j+1})\}_{j=1}^k$ 
of $(A,B)$ with $k\ll m$.

\subsection{Reconstruction of approximate eigenpairs and detection of convergence}
In this section, for brevity, we omit the subscript $j$ and denote 
by $\left(\theta,\tilde u,\tilde v\right)$ any generalized Ritz 
approximation of $(A,B)$ computed at the $m$th step of the GSSLBD method.
We reconstruct the approximations to the conjugate eigenpairs 
$(\lambda_{\pm},x_{\pm})$ of $(A,B)$ defined in \eqref{eigenpairs} by 
\begin{equation}\label{appeigenpairs}
	(\tilde\lambda_{\pm},\tilde x_{\pm})
	=\left(\pm\mi\theta,\frac{1}{\sqrt2}(\tilde u\pm\mi\tilde v)\right).
\end{equation} 
Their residuals are
\begin{equation}\label{defres}
	r_{\pm} = A\tilde x_{\pm}-\tilde\lambda_{\pm}B \tilde x_{\pm}
	= r_{\mathrm{R}} \pm \mi r_{\mathrm{I}}
	\qquad\mbox{with}\qquad
	\left\{\begin{aligned}
		r_{\mathrm{R}}&=\tfrac{1}{\sqrt2}(A\tilde u+\theta B\tilde v),\\
		r_{\mathrm{I}}&=\tfrac{1}{\sqrt2}(A\tilde v-\theta B\tilde u).
	\end{aligned}
	\right.
\end{equation}
Obviously, $r_{\pm}=\bm{0}$ if and only if $(\theta,\tilde u,\tilde v)$
is an exact generalized spectral  triplet of $(A,B)$ or, equivalently, 
$(\tilde\lambda_{\pm},\tilde x_{\pm})$ are exact eigenpairs of $(A,B)$.
By the square decomposition \eqref{squaredec} of $B$ and 
the definition \eqref{Achek} of $H$, it is easy to verify that 
\begin{equation}\label{res2}
	r_{\pm}=M(H{\tilde y}_{\pm}-\tilde \lambda_{\pm}{\tilde y}_{\pm})
	=M(\hat r_{\mathrm{R}}\pm\mi \hat r_{\mathrm{I}})
	\quad\mbox{with}\quad
	\left\{\begin{aligned}
		\hat r_{\mathrm{R}}&=\tfrac{1}{\sqrt2}(H\tilde w+\theta\tilde z)
		=M^{-1}r_{\mathrm{R}},\\
		\hat r_{\mathrm{I}}&=\tfrac{1}{\sqrt2}(H\tilde z-\theta\tilde w)
		=M^{-1}r_{\mathrm{I}}.
	\end{aligned}
	\right.
\end{equation}
where ${\tilde y}_{\pm}=\tfrac{1}{\sqrt2}
(\tilde w\pm\mi \tilde z)=M\tilde x_{\pm}$ 
is of $2$-norm unit length by the fact that $\tilde x_{\pm}$ 
is of $B$-norm unit length. 
Notice that $M=B^{\frac{1}{2}}$. 
If the residual norm of $(\theta,\tilde u,\tilde v)$ satisfies 
\begin{equation}\label{res}
	\|r_{\pm}\|=\sqrt{\|r_{\mathrm{R}}\|^2+
		\|r_{\mathrm{I}}\|^2}\leq \sqrt{\|B\|}\|H\|\cdot tol 
\end{equation}
for a user-prescribed tolerance $tol>0$, we claim that 
$(\tilde\lambda_{\pm},\tilde x_{\pm})$ has
converged and stop the iterations. 
We shall consider practical estimates for $\|B\|$ and 
$\|H\|$ in section~\ref{subsec:5}.  

Inserting $\tilde w=\widehat P_mc$ and 
$\tilde z=\widehat Q_md$ into \eqref{res2} and
making use of \eqref{LBDmat} and \eqref{Ritz}, 
it is easy to justify that
\begin{eqnarray}
	\|r_{\pm}\|&=&\frac{1}{\sqrt2}\sqrt{
		\|M(\widehat Q_m(G_m^\TT c-\theta d)
		+\beta_m\hat q_{m+1}e_{m,m}^\TT c)\|^2 
		+\|M\widehat P_m(G_md-\theta c)\|^2} \nonumber \\
	&=&\frac{1}{\sqrt2}\beta_m\cdot|e_{m,m}^\TT c|
	\cdot\|M\hat q_{m+1}\|
	=\frac{1}{\sqrt2}\beta_m\cdot|e_{m,m}^\TT c|
	\cdot\|B  q_{m+1}\|, \label{resieasy}
\end{eqnarray}
where the last equality holds since $\hat q_{m+1}=Mq_{m+1}$.
Therefore, one can calculate the residual norms of the generalized 
Ritz approximations of $(A,B)$ efficiently without explicitly 
forming the approximate generalized spectral vectors. 
We only compute them until the corresponding residual norms 
defined by \eqref{resieasy} drop below the right hand side of \eqref{res}. 
In addition, relation \eqref{resieasy} shows that if $\beta_m=0$, 
all the $m$ computed approximate generalized spectral triplets
$(\theta,\tilde u,\tilde v)$ are the exact ones of $(A,B)$. 
Analogously, we can draw the same conclusion for $\alpha_m=0$. 
In other words, once the $m$-step GSSLBD process breaks down, 
we can compute $m-1$ or $m$ exact conjugate eigenpairs of $(A,B)$, 
as we have mentioned at the end of Section~\ref{subsec:1}.

\section{A convergence analysis}\label{sec3}
For later use, denote $X_j=[u_j,v_j]$ and $Y_j=[w_j,z_j]$, 
$j=1,\dots,\ell$. 
Then $\XX_j=\spans\{X_j\}$ and $\YY_j=\spans\{Y_j\}$ 
are the eigenspaces of $(A,B)$ and $H$ defined by \eqref{Achek}
associated with the conjugate eigenvalues 
$\lambda_{\pm j}=\pm\mi\sigma_j$, 
respectively.  
In addition, we have $Y_j=MX_j$ and $\YY_j=M\XX_j$; 
see \eqref{spectralhatA}--\eqref{relations}. 

In order to make a convergence analysis on the GSSLBD method,
we review some preliminaries. 
For two subspaces $\WW$ and $\ZZ$ of $\mathbb{R}^{n}$, 
define the $B$-canonical angles between them by
\begin{equation}\label{Bcanonical}
	\angle_B(\mathcal{W},\mathcal{Z})
	=\angle(M\mathcal{W},M\mathcal{Z}),	
\end{equation}
where $M=B^{\frac{1}{2}}$ is as in \eqref{squaredec} and 
$\angle(\cdot,\cdot)$ denotes the general 
canonical angles between two subspaces with respect to $2$-norm 
\cite[pp.329--330]{golub2012matrix}. 
Let $\widetilde W$ and $\widetilde Z$ be the $B$-orthonormal basis 
matrices of $\WW$ and $\ZZ$, respectively, and $\widetilde W_{\perp_B}$ 
be such that $[\widetilde W,\widetilde W_{\perp_B}]$ is $B$-orthogonal.
Then by \eqref{squaredec} the columns of $M\widetilde W$, 
$M\widetilde W_{\perp_B}$ and $M\widetilde Z$
form orthonormal basis matrices of the subspaces $M\WW$, 
$\{M\WW\}^{\perp}$ and $M\ZZ$, respectively, where $\{\cdot\}^{\perp}$ 
denotes the orthogonality complement of a subspace 
with respect to $\mathbb{R}^n$. 
By the definition of the $2$-norm distance between two 
subspaces \cite{jia2001}, we defined the $2$-norm $B$-distance 
between $\mathcal{W}$ and $\mathcal{Z}$ by  
$$
\|\sin\angle_B(\mathcal{W},\mathcal{Z})\|
=\|(I-M\widetilde W\widetilde W^\TT M)(M\widetilde Z)\|
=\|\widetilde W_{\perp_B}^\TT  M^2\widetilde Z\|
=\|\widetilde W_{\perp_B}^\TT  B\widetilde Z\|.
$$
In this paper, we will use the Frobenius norm  
\begin{equation}\label{sinf}
	\|\sin\angle_B(\mathcal{W},\mathcal{Z})\|_{\FF}
	=\|\widetilde W_{\perp_B}^\TT  B\widetilde Z\|_{\FF},
\end{equation}
which equals to the square root of the squares sum of sines 
of all the $B$-canonical angles between $\WW$ and $\ZZ$, 
to measure the $B$-distance between those two subspaces.
Correspondingly, we denote by
$\|\tan\angle_B(\mathcal{W},\mathcal{Z})\|_{\FF}$ 
the square root of the squares sum of the tangents
of all the $B$-canonical angles between $\WW$ and $\ZZ$. 
These tangents are the generalized singular values of
the matrix pair $(\widetilde W_{\perp_B}^\TT B\widetilde Z, 
\widetilde W^\TT B\widetilde Z)$.
 
\begin{lemmamy}\label{lemma1}
	Let $M$ be as in \eqref{squaredec}. 
	For any two subspaces $\WW$ and $\ZZ$ of $\mathbb{R}^n$, 
	it holds that
\begin{equation}\label{sinWZ0}
	\frac{1}{\sqrt{\kappa(B)}}\|\sin\angle_B(\mathcal{W},\mathcal{Z})\|_{\FF}
	\leq\|\sin\angle(\mathcal{W},\mathcal{Z})\|_{\FF}
	\leq\sqrt{\kappa(B)}\|\sin\angle_B(\mathcal{W},\mathcal{Z})\|_{\FF},
\end{equation}
where $\kappa(B)=\|B\|\|B^{-1}\|$ is the condition number of $B$.
\end{lemmamy}

\begin{proof}
With the denotations described above, we have the following 
decomposition for $\widetilde Z$:
\begin{equation}\label{Edecomp}
\widetilde Z=\widetilde WE_1+\widetilde W_{\perp_B}E_2.
\end{equation} 
Premultiplying $\widetilde W_{\perp_B}^\TT B$ both sides of 
\eqref{Edecomp} delivers $E_2=\widetilde W_{\perp_B}^\TT B\widetilde Z$. 
Therefore, \eqref{sinf} amounts to
\begin{equation}\label{sinBWZ}
\|\sin\angle_B(\mathcal{W},\mathcal{Z})\|_{\FF}=\|E_2\|_{\FF}.
\end{equation}
Note that $\widetilde W$ and $\widetilde Z$ have full column ranks. 
The columns of $\widetilde W(\widetilde W^\TT \widetilde W)^{-1/2}$ 
and $\widetilde Z(\widetilde Z^\TT \widetilde Z)^{-1/2}$ 
form orthonormal bases of $\mathcal{W}$ and $\mathcal{Z}$, respectively. 
Assume that $\widetilde W_{\perp}$ is the orthonormal basis 
matrix of $\WW^{\perp}$, i.e.,
$[\widetilde W(\widetilde W^\TT\widetilde W)^{-1/2},\widetilde W_{\perp}]$
is orthogonal.
Then by definition, we have 
\begin{eqnarray}
	\|\sin\angle(\mathcal{W},\mathcal{Z})\|_{\FF}&=&
	\|\widetilde W_{\perp}^\TT\widetilde Z
	(\widetilde Z^\TT\widetilde Z)^{-1/2}\|_{\FF}
	=\|\widetilde W_{\perp}^\TT\widetilde W_{\perp_B}E_2 
	(\widetilde Z^\TT\widetilde Z)^{-1/2}\|_\FF \nonumber \\
	&=&\|(\widetilde W_{\perp}^\TT M^{-1}\widetilde W_{\perp_B}^{\prime})E_2 
	((\widetilde Z^{\prime})^\TT B^{-1}\widetilde Z^{\prime})^{-1/2}\|_\FF \nonumber \\ 
	&\leq&\|M^{-1}\|\cdot\|E_2\|_\FF\cdot\|((\widetilde Z^{\prime})^\TT 
	B^{-1}\widetilde Z^{\prime})^{-1/2} \| \nonumber \\
	&\leq&\|B^{-1}\|^{\frac{1}{2}}\cdot\|E_2\|_\FF\cdot\|B\|^{\frac{1}{2}}, \label{sinWZ}  
\end{eqnarray}
where the second equality follows by inserting \eqref{Edecomp} into the 
first one and realizing $\widetilde W_{\perp}^\TT\widetilde W=\bm{0}$,
the third equality and the first inequality result from denoting the orthonormal 
$\widetilde W_{\perp_B}^{\prime} = M\widetilde W_{\perp_B}$ 
and $\widetilde Z^{\prime}=M\widetilde Z$, and the last inequality holds 
by the Cauchy interlace theorem of eigenvalues and \eqref{squaredec}. 
As a consequence, we obtain the second inequality of \eqref{sinWZ0} 
by inserting \eqref{sinBWZ} 
and $\kappa(B)=\|B\|\|B^{-1}\|$ into the right hand side of \eqref{sinWZ}. 

Denote by $E_3$ the matrix in the right 
hand side of the third equality of \eqref{sinWZ}, where notice that 
$\widetilde W_{\perp_B}^{\prime}=M\widetilde W_{\perp_B}$, we have  
\begin{equation}\label{sinWZ2}
	\|\sin\angle(\mathcal{W},\mathcal{Z})\|_{\FF}=\|E_3\|_\FF
	\quad\mbox{with}\quad
	E_3=(\widetilde W_{\perp}^\TT \widetilde W_{\perp_B})E_2 
	((\widetilde Z^{\prime})^\TT B^{-1}\widetilde Z^{\prime})^{-1/2}.
\end{equation}
Since $[\widetilde W,\widetilde W_{\perp}]$ is nonsingular, 
we can decompose $\widetilde W_{\perp_B}=\widetilde WE_4+\widetilde W_{\perp}E_5$ 
for some $E_4$ and $E_5$.   
Premultiplying $\widetilde W_{\perp}^\TT$ and $\widetilde W_{\perp_B}^\TT B$ 
this equation, respectively, gives rise to  
$\widetilde W_{\perp}^\TT\widetilde W_{\perp_B}=
E_5=(\widetilde W_{\perp_B}^\TT B\widetilde W_{\perp})^{-1}
=((\widetilde W_{\perp_B}^{\prime})^\TT M\widetilde W_{\perp})^{-1}$. 
Inserting this relation into \eqref{sinWZ2}, rearranging the 
resulting equation and taking Frobenius norms, we obtain   
\begin{eqnarray}
	\|E_2\|_\FF 
	&=&\|( (\widetilde W_{\perp_B}^{\prime})^\TT M\widetilde W_{\perp})
	E_5((\widetilde Z^{\prime})^\TT B^{-1}\widetilde Z^{\prime})^{1/2}\|_\FF \nonumber \\
	&\leq &\|M\|\cdot\|	E_5\|_\FF \cdot 
	\|((\widetilde Z^{\prime})^\TT B^{-1}\widetilde Z^{\prime})^{1/2}\| \nonumber  \\
	&\leq &\|B\|^{\frac{1}{2}}\cdot\|E_5\|_\FF \cdot\|B^{-1}\|^{\frac{1}{2}}, \label{E2E3}
\end{eqnarray}
where we have used the fact that the columns of $\widetilde W_{\perp}$, 
$\widetilde W_{\perp_B}^{\prime}$ and $\widetilde Z^{\prime}$ are all
orthonormal. 
Then the first inequality of \eqref{sinWZ0} follows by 
inserting \eqref{sinBWZ}, \eqref{sinWZ2} and $\kappa(B)=\|B\|\|B^{-1}\|$ 
into \eqref{E2E3}.
\end{proof}

\begin{remark}
	Following the same derivations, we can easily prove 
	the result \eqref{sinWZ0} for $2$-norm. 
Lemma~\ref{lemma1} indicates the equivalent of the 
general and $B$-distances   
between $\WW$ and $\ZZ$. 
With a well or even mildly ill conditioned $B$,   
$\|\sin\angle(\mathcal{W},\mathcal{Z})\|_{\FF}
\sim\|\sin\angle_B(\mathcal{W},\mathcal{Z})\|_{\FF}$. 
Therefore, it has the same effect by measuring either 
distance between two subspaces.	
\end{remark}

Note that the real and imaginary parts of the eigenvectors $x_{\pm j}$ 
of $(A,B)$ can be switched because $\mi x_{\pm j}$ are also its 
eigenvectors and the left and right generalized spectral vectors 
$u_j$ and $v_j$  are the right and left ones, too; 
see \eqref{spectralAB}.
As we have shown, any pair of the approximate left and right 
generalized spectral vectors $\tilde u_{j^{\prime}}$ and 
$\tilde v_{j^{\prime}}$ extracted from the $B$-biorthogonal 
subspaces $\UU_m$ and $\VV_m$ are mutually $B$-orthogonal, 
using which we construct the real and imaginary parts of 
the approximation to an eigenvector, say 
$x_{\pm j}=u_j\pm v_j$, of $(A,B)$. 
This means, the $2$-dimensional subspace 
$\spans\{\tilde u_{j^{\prime}},\tilde v_{j^{\prime}}\}$
form a reasonable approximation to the eigenspace 
$\XX_j=\spans\{x_j,x_{-j}\}$ corresponding to 
the conjugate eigenvalues $\lambda_{\pm j}$ of $(A,B)$. 
Therefore, in the eigenvalue computation context of a
skew-symmetric/symmetric positive definite matrix pair, 
because of these properties, when analyzing the convergence behavior 
of the GSSLBD method, we consider the $B$-norm direct sum 
$\UU_m\oplus_B\VV_m$ as a whole, and estimate the distance 
between a desired two dimensional eigenspace $\XX_j$ of $(A,B)$ 
and the $2m$-dimensional $\UU_m\oplus_B\VV_m$ for $j$ small. 
Exploiting these distances and their estimates, we can 
establish a priori error bounds for the approximate eigenvalues 
and eigenspaces computed by the GSSLBD method, showing how 
fast they converge to zero when computing several extreme 
eigenpairs.

In terms of the definition of
$\|\tan\angle_B(\cdot,\cdot)\|_{\FF}$,
we present the estimate for
$\|\tan\angle_B(\UU_m\oplus_B\VV_m,\mathcal{X}_j)\|_{\FF}$, 
which is a generalization of Theorem~4.1 in \cite{huang2024skew}.

\begin{theoremmy}\label{thm2}
	Let $\UU_m$ and $\VV_m$ be defined by \eqref{defUV},
	and suppose that $X_j^\TT  BF$ with 
	$F=[p_1,q_1]$ is nonsingular.
	Then the following estimate holds for any integer $1\leq j\leq m$:
	\begin{equation}\label{accBUV}
		\|\tan\angle_B(\UU_m\oplus_B\VV_m,\mathcal{X}_j)\|_{\FF}\leq
		\frac{\eta_j}{\chi_{m-j}(\xi_j)}\|\tan\angle_B(\mathcal{F},
		\mathcal{X}_j)\|_{\FF},
	\end{equation}
	where $\mathcal{F}=\spans\{p_1,q_1\}$, $\chi_{i}(\cdot)$ 
	is the degree $i$ Chebyshev
	polynomial of the first kind, and
	\begin{equation}\label{gammarho}
		\xi_j=1+2\cdot\frac{\sigma_j^2-\sigma_{j+1}^2}
		{\sigma^2_{j+1}-\sigma^2_{\ell}}
		\quad\mbox{and}\quad
		\eta_j=\left\{\begin{aligned}
			&1,\qquad \quad\qquad\qquad\mbox{if} \quad  j=1,\\
			&\prod\limits_{i=1}^{j-1}\frac{\sigma_i^2-\sigma_\ell^2}
			{\sigma_i^2-\sigma_j^2}, \hspace{2.4em}\mbox{if}\quad j>1.
		\end{aligned}\right.
	\end{equation}
Moreover, provided that $X_{\ell-j+1}^\TT BF$ is nonsingular, 
the following bounds hold for  $1\leq j\leq m$:  
\begin{equation}\label{accBUVs}
	\vspace{-0.5em}
	\|\tan\angle_B(\UU_m\oplus_B\VV_m,\mathcal{X}_{\ell-j+1})\|_{\FF}\leq
	\frac{{\eta}_j^{\prime}}{\chi_{m-j}({\xi}_j^{\prime})}
	\|\tan\angle_B(\mathcal{F},
	\mathcal{X}_{\ell-j+1})\|_{\FF},
\end{equation}
where\vspace{-0.5em}
\begin{equation}\label{gammarho2}
	{\xi}_j^{\prime}=1+2\cdot\frac{\sigma_{\ell-j+1}^2-
		\sigma_{\ell-j}^2}{\sigma^2_{\ell-j}-\sigma^2_1}
	\quad\mbox{and}\quad
	{\eta}_j^{\prime}=\left\{\begin{aligned}
		&1,\hspace{9.75em} \mbox{if} \quad  j=1,\\
		&\prod\limits_{i=1}^{j-1}\frac{\sigma_{\ell-i+1}^2-\sigma_1^2}
		{\sigma_{\ell-i+1}^2-\sigma_{\ell-j+1}^2},
		\hspace{1.85em}\mbox{if}\quad j>1.
	\end{aligned}\right.	
\end{equation}
\end{theoremmy}

\begin{proof} 
Recall from \eqref{WWZZm} that $\WW_m$ and $\ZZ_m$ are the 
biorthogonal left and right Lanczos subspaces generated by 
the $m$-step SSLBD process applied to $H$ defined by \eqref{Achek} 
with the initial unit-length vector $\hat q_1=Mq_1$. 
For the concerned eigenspace $\YY_j=M\XX_j$ of $H$ 
corresponding to the conjugate eigenvalues $\lambda_{\pm j}$ 
with the orthonormal basis matrix $Y_j=MX_j$, 	
the matrix $Y_j^\TT\widehat F=X_j^\TT BF$ with 
$\widehat F=[\hat p_1,\hat q_1]=MF$ is nonsingular by assumption.  
Then Theorem~4.1 of \cite{huang2024skew} indicates 
that for any integer $1\leq j\leq m$,
\begin{equation}\label{tanWZ} 
	\|\tan\angle(\WW_m\oplus\ZZ_m,{\mathcal{Y}}_j)\|_{\FF}\leq
	\frac{\eta_j}{\chi_{m-j}(\xi_j)}
	\|\tan\angle( \widehat {\mathcal{F}}, {\mathcal{Y}_j})\|_{\FF},
\end{equation} 
where $\widehat {\mathcal{F}}=\spans\{\widehat F\}=M\mathcal{F}$, 
$\xi_j$ and $\eta_j$ are defined by \eqref{gammarho} and 
$\chi_{i}(\cdot)$ is the degree $i$ Chebyshev
polynomial of the first kind. 
Therefore, we obtain \eqref{accBUV} by realizing from the definition 
that $\angle(\WW_m\oplus\ZZ_m,\mathcal{Y}_j)
=\angle_B(\UU_m\oplus_B\VV_m,\mathcal{X}_j)$ 
and
$\angle(\widehat{\mathcal{F}},\mathcal{Y}_j)
=\angle_B(\mathcal{F},\mathcal{X}_j)$. 

Following the same derivations as above, we can prove \eqref{accBUVs}. 
\end{proof}

\begin{remark}\label{remark3}
	Following the same derivations, we can 
	straightforwardly prove \eqref{accBUV} 
	and \eqref{accBUVs} with the Frobenius norms replaced by $2$-norms.
	 Theorem~\ref{thm2} establishes accuracy estimates for
	 $\mathcal{X}_j$ and $\mathcal{X}_{\ell-j+1}$ approaching
	 the subspace $\UU_m\oplus_B\VV_m$ as the GSSLBD 
	 process proceeds. 
	 The constants $\eta_j\geq 1$, $\xi_j>1$ and $\eta_j^{\prime}\geq1$, 
	 $\xi_j^{\prime}>1$ defined by \eqref{gammarho} and \eqref{gammarho2} 
	 depend only on the spectral distribution of $(A,B)$.  
	 (\romannumeral1) For a fixed $j\ll m$, as long as the initial 
	 subspace $\mathcal{F}$ contains some information on 
	 $\XX_j$ or $\XX_{\ell-j+1}$, i.e., $\XX_j^\TT BF$ or 
	 $\XX_{\ell-j+1}^\TT BF$ is nonsingular, 
	 then $\|\tan\angle_B(\mathcal{F},\XX_j)\|_{\FF}$ or
	 $\|\tan\angle_B(\mathcal{F},\XX_{\ell-j+1})\|_{\FF}$ is finite 
	 in \eqref{accBUV} or \eqref{accBUVs}.
	 (\romannumeral2) The larger $m$ is, the smaller the parameters 
	 $\eta_j/\chi_{m-j}(\xi_j)$ and $\eta_j^{\prime}/\chi_{m-j}(\xi_j^{\prime})$ 
	 are, and therefore the closer $\XX_j$ and $\XX_{\ell-j+1}$ are 
	 to $\UU_m\oplus_B\VV_m$ regarding $B$-distance. 
	 (\romannumeral3) The better a concerned $\sigma_j$ or 
	 $\sigma_{\ell-j+1}$ is separated from the other 
	 generalized spectra values of $(A,B)$, 
	 the smaller $\eta_j$ or $\eta_{\ell-j+1}$ and the larger $\xi_j$ 
	 or $\xi_{\ell-j+1}$ are, meaning that the faster $\XX_j$ or 
	 $\XX_{\ell-j+1}$ approaches $\UU_m\oplus_B\VV_m$ as $m$ increases. 
	 Generally, $\|\tan\angle_B(\UU_m\oplus_B\VV_m,\XX_j)\|_{\FF}$ and  
	 $\|\tan\angle_B(\UU_m\oplus_B\VV_m,\XX_{m-j+1})\|_{\FF}$
	 decays faster for $j$ smaller. 
	 Therefore, $\UU_m\oplus_B\VV_m$ favors $\XX_j$ and 
	 $\XX_{\ell-j+1}$ for $j$ small.  
\end{remark}

\begin{remark}
	If $\kappa(X_j^\TT BF) $ or $\kappa(X_{\ell-j+1}^\TT BF)=1$ 
	for some $1\leq j\leq m$, then 
	$\|\tan\angle_B(\mathcal{F},\XX_j)\|_{\FF}$ or 
	$\|\tan\angle_B(\mathcal{F},\XX_{\ell-j+1})\|_{\FF}=0$, 
	i.e., $\mathcal{F}=\XX_j$ or $\XX_{\ell-j+1}$. 
	In this case, Algorithm~\ref{alg2} breaks down at the first step 
	and we find the exact $u_j$, $v_j$ or $u_{\ell-j+1}$, $v_{\ell-j+1}$. 
	On the other hand,  if $X_j^\TT BF$ or $X_{\ell-j+1}^\TT BF$ is singular for 
	some $1\leq j\leq m$, then 
	$\|\tan\angle_B(\mathcal{F},\XX_j)\|_{\FF}$ or 
	$\|\tan\angle_B(\mathcal{F},\XX_{\ell-j+1})\|_{\FF}\\=+\infty$, that is,  
	the initial $\mathcal{F}$ is deficient in $\XX_j$ or $\XX_{\ell-j+1}$. 
	In this case, $\UU_m\oplus_B\VV_m$ contains no information on 
	$\XX_j$ or $\XX_{\ell-j+1}$.
	As a consequence, one cannot find any reasonable approximation 
	to $u_j$, $v_j$ or $u_{\ell-j+1}$, $v_{\ell-j+1}$ from $\UU_m$ and
	$\VV_m$ for any $m<\ell$.
\end{remark}
  
Next, we present a priori accuracy estimates for the generalized 
Ritz approximations of $(A,B)$ computed by the GSSLBD method 
in term of $\|\sin\angle_B(\UU_m\oplus_B\VV_m,\mathcal{X}_j)\|_{\FF}$.  

\begin{theoremmy}\label{thm3}
For $1\leq j\leq \ell$, assume that  
$(\theta_{j^{\prime}},\tilde u_{j^{\prime}},\tilde v_{j^{\prime}})$
is a generalized Ritz approximation to the generalized 
spectral triplet $(\sigma_j,u_j,v_j)$ of $(A,B)$ with 
$\theta_{j^{\prime}}$ closest to $\sigma_j$
among all the Ritz values.
Denote 
$\widetilde{\XX}_{j^{\prime}}=\spans\{\tilde u_{j^{\prime}},
\tilde v_{j^{\prime}}\}$. Then
	\begin{equation}\label{accbuv}
	\|\sin\angle_B(\widetilde{\XX}_{j^{\prime}},\XX_j)\|_{\FF}\leq
	\sqrt{1+\frac{\|\PP_mH(I-\PP_m)\|^2}
		{\min_{i\neq j^{\prime}}|\sigma_j-\theta_i|^2}}
	\|\sin\angle_B(\UU_m\oplus_B\VV_m,\mathcal{X}_j)\|_{\FF},	
	\end{equation}
where $\PP_m$ is the orthogonal projector onto $\WW_m\oplus\ZZ_m$ 
with $\WW_m$ and $\ZZ_m$ defined by \eqref{WWZZm}.
\end{theoremmy}

\begin{proof}
From \eqref{standard}, $(\theta_{j^{\prime}},\tilde w_{j^{\prime}}
=M\tilde u_{j^{\prime}},\tilde z_{j^{\prime}}=M\tilde v_{j^{\prime}})$ 
is the Ritz approximation to the singular triplet $(\sigma_j,w_j,z_j)$ 
of $H$ with respect to $\WW_m$ and $\ZZ_m$. 
Denote $\widetilde \YY_{j^{\prime}}=
\spans\{\tilde w_{j^{\prime}},\tilde z_{j^{\prime}}\}
=M\widetilde\XX_{j^{\prime}}$. 
Then Theorem~4.4 of \cite{huang2024skew} indicates that 
\begin{equation}\label{accuv}
	\|\sin\angle(\widetilde\YY_{j^{\prime}},\YY_j)\|_{\FF}\leq
	\sqrt{1+\frac{\|\PP_mH(I-\PP_m)\|^2}
		{\min_{i\neq j^{\prime}}|\sigma_j-\theta_i|^2}}
	\|\sin\angle(\WW_m\oplus\ZZ_m,{\mathcal{Y}}_j)\|_{\FF},	
\end{equation}
where $\PP_m$ is the orthogonal projector onto  
$\WW_m\oplus\ZZ_m=M(\UU_m\oplus_B\VV_m)$. 
Therefore, we obtain \eqref{accbuv} by realizing    
$\angle_B(\widetilde\XX_{j^{\prime}},\XX_j)
\!=\!\angle(\widetilde\YY_{j^{\prime}},\YY_j)$ and  
$\angle_B(\UU_m\!\oplus_B\!\VV_m,\mathcal{X}_j)
\!=\!\angle(\WW_m\!\oplus\!\ZZ_m, {\mathcal{Y}}_j)$.
\end{proof}

\begin{remark}
Theorem~\ref{thm3} extends the a priori error bound for Ritz blocks 
in the real skew-symmetric eigenproblem case 
(cf.~\cite[Theorem~4.1]{huang2024skew}) to the real 
skew-symmetric/symmetric positive 
definite generalized eigenproblem case. 
Combining Theorems~\ref{thm2}--\ref{thm3} and Remark~\ref{remark3}, 
we conclude  that the GSSLBD method favors the eigenspaces
$\XX_j$ and $\XX_{\ell-j+1}$ for $j$ small. 
Moreover, if $\sigma_j$ is well separated from the other Ritz values  
than the $\theta_{j^{\prime}}$, i.e., 
$\min_{i\neq j^{\prime}}|\sigma_j-\theta_i|$ is not small,
then Theorem~\ref{thm3} guarantees that the approximate 
eigenspace $\widetilde\XX_{j^{\prime}}$ 
converges to the exact $\XX_j$ as fast as the latter 
approaches the subspace $\UU_m\oplus_B\VV_m$.
\end{remark}

\begin{theoremmy}\label{thm4}
With the notations of Theorem~\ref{thm3}, assume that the $B$-angles
\begin{equation}\label{assump}
\angle_B(\tilde{u}_{j^{\prime}},u_j)\leq \frac{\pi}{4}
\mbox{\ \ \ and\ \ \ }
\angle_B(\tilde{v}_{j^{\prime}},v_j)\leq \frac{\pi}{4}.
\end{equation}
Then 
\begin{eqnarray}
|\theta_{j^{\prime}}-\sigma_j|&\leq& \sigma_j
\|\sin\angle_B(\widetilde \XX_{j^{\prime}},\mathcal{X}_j)\|^2+
\frac{\sigma_1}{\sqrt{2}}
\|\sin\angle_B(\widetilde \XX_{j^{\prime}},\mathcal{X}_j)\|_{\FF}^2,
\label{ritzerror}\\
|\theta_{j^{\prime}}-\sigma_j|&\leq& (\sigma_1+\sigma_j)
\|\sin\angle_B(\widetilde \XX_{j^{\prime}},\mathcal{X}_j)\|^2. \label{ritzerror2}
\end{eqnarray}
\end{theoremmy}

\begin{proof}
By assumption, $\angle(\tilde w_{j^{\prime}},w_j)
=\angle_B(\tilde u_{j^{\prime}},u_j)\leq\frac{\pi}{4}$ 
and $\angle(\tilde z_{j^{\prime}},z_j)
=\angle_B(\tilde v_{j^{\prime}},v_j)\leq\frac{\pi}{4}$. 
For the accuracy of $\theta_{j^{\prime}}$ as an approximation to 
the singular value $\sigma_j$ of $H$, Theorem~4.5 of 
\cite{huang2024skew} shows that
\begin{eqnarray*}
	|\theta_{j^{\prime}}-\sigma_j|&\leq& \sigma_j
	\|\sin\angle(\widetilde \YY_{j^{\prime}}, {\mathcal{Y}}_j)\|^2+
	\frac{\sigma_1}{\sqrt{2}}
	\|\sin\angle(\widetilde \YY_{j^{\prime}}, {\mathcal{Y}}_j)\|_{\FF}^2, \\
	|\theta_{j^{\prime}}-\sigma_j|&\leq& (\sigma_1+\sigma_j)
	\|\sin\angle(\widetilde\YY_{j^{\prime}}, {\mathcal{Y}}_j)\|^2, 
\end{eqnarray*}
which are equivalent to \eqref{ritzerror}--\eqref{ritzerror2} since
$\angle(\widetilde\YY_{j^{\prime}},{\mathcal{Y}}_j)
=\angle_B(\widetilde\XX_{j^{\prime}},\mathcal{X}_j)$.
\end{proof}

\begin{remark}
	The results of Theorem~\ref{thm7} rely on the 
	assumption \eqref{assump}, which is weak and met soon as $m$ increases. 
	Note that $\|\sin\angle_B(\widetilde\XX_{j^{\prime}},\mathcal{X}_j)\|
	\leq\|\sin\angle_B(\widetilde\XX_{j^{\prime}},\mathcal{X}_j)\|_{\FF}$ 
	in \eqref{ritzerror}--\eqref{ritzerror2}. 
	Therefore, as $m$ increases, the Ritz value $\theta_{j^{\prime}}$ 
	converges to the generalized spectral value $\sigma_j$ of $(A,B)$ 
	twice as fast as the relevant approximate eigenspace 
	$\widetilde \XX_{j^{\prime}}$ approaches the exact $\XX_j$.
\end{remark}

\section{An implicitly restarted GSSLBD algorithm with partial 
	$B$-reorthogonalization}\label{sec:4} 
In this section, we introduce an efficient partial 
reorthogonalization strategy for the GSSLBD process and 
develop an implicitly restarted GSSLBD algorithm.

\subsection{Partial $B$-reorthogonalization}\label{subsec:3} 
In computations, as $m$ increases, 
the left and right generalized Lanczos vectors generated 
by the GSSLBD process loss the $B$-orthogonality and 
$B$-biorthogonality gradually, leading to the ``ghost'' 
phenomena of computed Ritz approximations and delayed 
and even failed convergence. 
For the sake of the numerical reliability and robustness 
of the GSSLBD method, appropriate reorthogonalization 
is absolutely vital.

\begin{theoremmy}\label{thm7} 
Let the bidiagonal $G_m$ and $P_m=[p_1,\ldots,p_m]$, 
$Q_m=[q_1,\ldots,q_{m}]$ be generated by the $m$-step GSSLBD process.   
Provided that $\alpha_i,\ i=1,2,\ldots,m$ and $\beta_i,\ i=1,2,\ldots,m-1$
are not small, and
	\begin{equation}\label{semiorth}
		\max\Big\{\max_{1\leq i<j\leq m} \left|p_i^\TT Bp_j\right|,
		\max_{1\leq i<j\leq m} \left|q_i^\TT Bq_j\right|,
		\max_{1\leq i , j\leq m} \left|p_i^\TT Bq_j\right|\Big\}
		\leq\sqrt{\frac{\epsilon}{m}},
	\end{equation}
where $\epsilon$ is the machine precision.
Then there exist two orthonormal and biorthogonal matrices 
$\widetilde P_m\in\mathbb{R}^{n\times m}$ 
and $\widetilde Q_m\in\mathbb{R}^{n\times m}$ such that 
\begin{equation}\label{pqm}
	\widetilde P_m^\TT H\widetilde Q_{m}=  G_m+\Delta_m,
\end{equation}
where the elements of $\Delta_m$ are $\mathcal{O}(\|H\|\epsilon)$ 
in size with $H$ defined by \eqref{Achek}. 
\end{theoremmy}

\begin{proof}
Equivalent to the $m$-step GSSLBD process applied to $(A,B)$, 
the $m$-step SSLBD process performed on $H$ produces the same 
bidiagonal matrix $G_m$ and the Lanczos vectors
$\hat p_1,\dots,\hat p_m$ and $\hat q_1,\dots,\hat q_m$ with 
$\hat p_j=Mp_j$ and $\hat q_j=Mq_j$, $j=1,\dots,m$; see section~\ref{subsec:1}.   
By \eqref{squaredec}, bound \eqref{semiorth} indicates that 
\begin{equation*} 
	\max\Big\{\max_{1\leq i<j\leq m} \left|\hat p_i^\TT \hat p_j\right|,
	\max_{1\leq i<j\leq m} \left|\hat q_i^\TT \hat q_j\right|,
	\max_{1\leq i , j\leq m} \left|\hat p_i^\TT \hat q_j\right|\Big\}
	\leq\sqrt{\frac{\epsilon}{m}}.
\end{equation*} 
Then Theorem~5.1 of \cite{huang2024skew} shows that there exists 
two orthonormal and biorthogonal matrices
 $\widetilde P_m\in\mathbb{R}^{n\times m}$ and 
$\widetilde Q_m\in\mathbb{R}^{n\times m}$ that satisfy 
$\spans\{\widetilde P_m,\widetilde Q_m\}
=\spans\{\hat p_1,\dots,\hat p_m,\hat q_1,\dots,\hat q_m\}$ 
such that \eqref{pqm} is fulfilled. 
\end{proof}

\begin{remark}
	Theorem~\ref{thm7} illustrates that if the generalized 
	Lanczos vectors computed by the $m$-step GSSLBD process 
	are made numerically $B$-orthogonal and $B$-biorthogonal 
	to the level of $\mathcal{O}(\sqrt{\varepsilon})$, 
	then up to roundoff errors the 
	upper bidiagonal matrix $G_m$ is equivalent to the projection 
	matrix $\widetilde P_m^\TT H\widetilde Q_{m}$ of $H$ with 
	respect to the subspaces $\spans\{\widetilde P_m\}$ and 
	$\spans\{\widetilde{Q}_m\}$. 
	As a consequence, as approximations to the singular values of 
	$H$ and generalized spectral values of $(A,B)$, 
	the singular values of $G_m$, i.e., the computed Ritz values, 
	are as accurate as those of $\widetilde P_m^\TT H\widetilde Q_{m}$, i.e., the exact ones, 
	within the error $\mathcal{O}(\|H\|\epsilon)$.
	In other words, the semi-$B$-orthogonality and 
	semi-$B$-biorthogonality of the generalized 
	Lanczos vectors suffice to make the computed and true Ritz 
	values equally accurate in finite precision arithmetic. 
	On the other hand, the full numerical $B$-orthogonality
	and $B$-biorthogonality at the level of $\mathcal{O}(\varepsilon)$ 
	are unnecessary and wasteful for the accurate computation 
	of eigenvalues of $(A,B)$. 
\end{remark}

In what follows, we generalize the two-sided partial reorthogonalization 
strategy, which is commonly adopted in the LBD type SVD methods 
\cite{jia2003implicitly,jia2010refined,larsen2001combining} 
and extended to the SSLBD process for the real skew-symmetric 
eigenvalue and SVD problems \cite{huang2024skew}, 
to the GSSLBD process so as to maintain the desired numerical 
semi-$B$-orthogonality and semi-$B$-biorthogonality of the 
left and right generalized Lanczos vectors. 
To this end, we need to estimate the levels 
of $B$-orthogonality and $B$-biorthogonality of the left and right 
generalized Lanczos vectors to determine whether \eqref{semiorth} 
is fulfilled. 
Once \eqref{semiorth} is wrecked, we use an efficient partial 
$B$-reorthogonalization module to reinstate \eqref{semiorth}. 
Motivated by the recurrences of Larsen \cite{larsen1998lanczos} 
designed for the LBD process and the variant extended to the 
SSLBD process \cite{huang2024skew}, we make full use of the 
GSSLBD process to achieve this goal. 

Specifically, denote by  $\Phi\in\mathbb{R}^{m\times m}$ and 
$\Psi\in\mathbb{R}^{(m+1)\times(m+1)}$ the estimate matrices 
of the level of $B$-orthogonality between the left and right 
generalized Lanczos vectors, respectively, and by 
$\Omega\in\mathbb{R}^{m\times (m+1)}$ the estimate matrix of 
the level of $B$-biorthogonality of these two sets of vectors. 
In other words, their $(i,j)$-elements $\varphi_{ij}\approx p_i^\TT Bp_j$,
$\psi_{ij}\approx q_i^\TT Bq_j$ and $\omega_{ij}\approx p_i^\TT Bq_j$. 
Note that $\varphi_{jj}=\psi_{jj}=1$, $j=1,\dots,m+1$.  
With a fixed $1\leq j\leq m$, by \eqref{LBDmat2},  
we have 
\begin{eqnarray}
\alpha_j\varphi_{ij}=p_i^\TT B(\alpha_jp_j)
&=&p_i^\TT B(B^{-1}Aq_j-\beta_{j-1}p_{j-1})   
= p_i^\TT Aq_j-\beta_{j-1}p_i^\TT Bp_{j-1}  \nonumber \\
&=&(\alpha_iBq_i+\beta_iBq_{i+1})^\TT q_j
-\beta_{j-1}\varphi_{i,j-1} \nonumber \\
&=&\alpha_i\psi_{ij}+\beta_i\psi_{i+1,j}
-\beta_{j-1}\varphi_{i,j-1}, \qquad\qquad\quad (1\leq i<j) \label{BorthPP} \\
\alpha_j\omega_{ji}	=(\alpha_jp_j)^\TT Bq_i
&=&(B^{-1}Aq_j-\beta_{j-1}p_{j-1})^\TT Bq_i
=-q_j^\TT Aq_i-\beta_{j-1}p_{j-1}^\TT Bq_i \nonumber \\
&=&-q_j^\TT (\beta_{i-1}Bp_{i-1}+\alpha_iBp_i)
-\beta_{j-1}\omega_{j-1,i} \nonumber \\
&=&-\beta_{i-1}\omega_{i-1,j}-\alpha_i \omega_{ij}
-\beta_{j-1} \omega_{j-1,i}, \qquad\quad\  (1\leq i\leq j) \label{BorthPQ}
\end{eqnarray} 
where $\omega_{0,1}=\cdots=\omega_{0,j}=0$. 
Following the same derivations as above, we obtain  
\begin{eqnarray}
\qquad\qquad\quad\beta_{j}\psi_{i,j+1}&=&\beta_{i-1}\varphi_{i-1,j}
+\alpha_i\varphi_{ij}-\alpha_j\psi_{i,j},  \qquad\qquad\qquad (1\leq i\leq j)  \label{BorthQQ}\\ 
\qquad\qquad\quad\beta_{j}\omega_{i,j+1}&=& -\beta_i \omega_{j,i+1}
-\alpha_i\omega_{ji}-\alpha_j\omega_{ij},  \qquad\qquad\qquad\  (1\leq i\leq j) \label{BorthQP}
\end{eqnarray}
where $\varphi_{0,i}=0$, $i=1,\dots,j$. 
Exploiting the relationships between the elements in the 
estimate matrices $\Phi$, $\Psi$ and $\Omega$ shown in 
\eqref{BorthPP}--\eqref{BorthQP}, we can efficiently 
calculate these quantities.  

Since $\Phi$ and $\Psi$ are symmetric with diagonal elements being one,
it suffices to compute their strictly upper triangular parts.
Concretely, at the $j$th step of the $m$-step GSSLBD process, 
when a generalized left Lanczos vector $p_j$ is computed, 
based on \eqref{BorthPP}--\eqref{BorthPQ}, 
we calculate the scaled new elements in $\Phi$ and $\Omega$ by 
\begin{eqnarray}
	&&\varphi_{ij}^{\prime}=\alpha_{i}\psi_{ij}+\beta_i
	\psi_{i+1,j}-\beta_{j-1}\varphi_{i,j-1},
	\hspace{3.85em}  
	\varphi_{ij}^{\prime} \hspace{0.11em}\circeq 
	\varphi_{ij}^{\prime}\hspace{0.11em}+
	\sign(\varphi_{ij}^{\prime}\hspace{0.11em})\tilde\epsilon, \label{orthPP} \\
	&&\omega_{ji}^{\prime}=-\alpha_{i}\omega_{ij}
	-\beta_{i-1}\omega_{i-1,j}
	-\beta_{j-1}\omega_{j-1,i},
	\qquad  
	\omega_{ji}^{\prime}\circeq\omega_{ji}^{\prime}
	+\sign(\omega_{ji}^{\prime})\tilde\epsilon, \label{orthPQ}
\end{eqnarray}
where $\circeq$ denotes the update operation, $\sign(\cdot)$ 
is the sign function and 
\begin{equation}\label{varepsilon}
	\tilde{\varepsilon}=\frac{\sqrt{n}\kappa(B)\|H\|}{2}\cdot \epsilon. 
\end{equation} 
Then by \eqref{BorthPP}--\eqref{BorthPQ}, we have $\varphi_{ij}=\varphi_{ij}^{\prime}/\alpha_j$ 
for $1\leq i<j$ 
and $\omega_{ji}=\omega_{ji}^{\prime}/\alpha_j$ for $1\leq i\leq j$. 
Particularly, the $\sign(\cdot)\tilde{\varepsilon}$ terms, 
as done in \cite{huang2024skew,larsen2001combining}, take the 
rounding errors of GSSLBD process into consideration,  
making the resulting $\varphi_{ij}$ and $\omega_{ji}$ 
as large in magnitude as possible and reliable as estimates 
for $p_j^\TT Bp_i$ and $p_j^\TT Bq_i$. 
Basing on the sizes of $\varphi_{ij}$ and $\omega_{ji}$, 
or equivalently, those of $\varphi_{ij}^{\prime}$ and 
$\omega_{ji}^{\prime}$, we determine the index sets 
\begin{equation}\label{orthPPQ}
	\mathcal{I}_{P,P}^{(j)}=\left\{1\leq i< j \big|
	 |\varphi_{ij}^{\prime}|\geq\alpha_j\sqrt{\tfrac{\epsilon}{m}}\right\}, \qquad
	\mathcal{I}_{P,Q}^{(j)}=\left\{1\leq i\leq j\big| 
	 |\omega_{ji}^{\prime}|\geq\alpha_j\sqrt{\tfrac{\epsilon}{m}}\right\},   
\end{equation}
which respectively contain the indices of previous left and right generalized 
Lanczos vectors $p_i$ and $q_i$ that $p_{j}$ has lost the 
semi-$B$-orthogonality and semi-$B$-biorthogonality to. 
Therefore, we generalize the modified Gram--Schmidt (MGS) orthogonalization
procedure \cite{golub2012matrix,saad2003iterative,stewart1998matrix} to the 
$B$-norm case and use the generalized MGS procedure to $B$-reorthogonalize 
$p_j$ first against those $p_i$ with $i\in \mathcal{I}_{P,P}^{(j)}$ and then 
against those $q_i$ with $i\in \mathcal{I}_{P,Q}^{(j)}$. 
 
Particularly, once $p_j$ is $B$-reorthogonalized to a left or right 
generalized Lanczos vector, the quantities $\varphi_{ij}$ and 
$\omega_{ji}$ change correspondingly and therefore need to be updated. 
Fortunately, this can be done efficiently during the generalized MGS 
reorthogonalization. Algorithm~\ref{alg3} summarizes the partial 
$B$-reorthogonalization process for so far, which is inserted 
between steps \ref{skwbld1} and \ref{skwbld2} of Algorithm~\ref{alg2}.  
Specially, we do not form $\varphi_{ij}$ and $\omega_{ji}$ 
explicitly until all the $B$-reorthogonalizations have been implemented  
so as to avoid unnecessary wast of computational cost. 
This is precisely the reason we use $\varphi_{ij}^{\prime}$ 
and $\omega_{ji}^{\prime}$ rather than 
$\varphi_{ij}$ and $\omega_{ji}$
to equivalently determine the index sets $\mathcal{I}_{P,P}^{(j)}$ 
and $\mathcal{I}_{P,Q}^{(j)}$ in \eqref{orthPPQ}. 

\begin{algorithm}[tbph]
	\caption{Partial $B$-reorthogonalization in the GSSLBD process-Part I.}\label{alg3}
	\begin{itemize}[leftmargin = 2em]
		\item[\footnotesize (\ref{skwbld1}.1)\label{reorth31}]
		Compute $\varphi_{1j}^{\prime},\varphi_{2j}^{\prime},\dots,\varphi_{j-1,j}^{\prime}$
		and $\omega_{j1}^{\prime},\omega_{j2}^{\prime},\dots,\omega_{jj}^{\prime}$ exploiting
		\eqref{orthPP}--\eqref{orthPQ},  and determine 
		the index sets $\mathcal{I}_{P,P}^{(j)}$ and $\mathcal{I}_{P,Q}^{(j)}$
		according to \eqref{orthPPQ}.
		
		\item[\footnotesize (\ref{skwbld1}.2)] \textbf{for}
		$i\in\mathcal{I}_{P,P}^{(j)}$ \textbf{do}
		\hspace{12em}  \textcolor{gray}{$\%$ partially 
			$B$-reorthogonalize $p_j$ to $P_{j-1}$}
		\begin{itemize}[leftmargin=1.5em]
			\item[] Compute $\tau=p_i^\TT Bs_j$, overwrite $s_j\circeq s_j-\tau p_i$ 
			and recalculate $\varphi_{ij}^{\prime}=p_i^\TT Bs_j$. 
			\item[] Update $\varphi_{lj}^{\prime} \circeq
			\varphi_{lj}^{\prime}-\tau\varphi_{li}$ ($1\leq l<j$, $l\neq i$)  
			and $\omega_{jl}^{\prime}= \omega_{jl}^{\prime}
			-\tau\omega_{il}$ ($1\leq l\leq j$).
		\end{itemize}
		\item[] \textbf{end for}
		
		\item[\footnotesize (\ref{skwbld1}.3)] \textbf{for}
		$i\in\mathcal{I}_{P,Q}^{(j)}$ \textbf{do}  \hspace{12em}
		\textcolor{gray}{$\%$ partially $B$-reorthogonalize $p_j$ to $Q_{j}$}
		\begin{itemize}[leftmargin=1.5em]
			\item[] Compute $\tau =q_i^\TT Bs_j$, overwrite $s_j=s_j-\tau q_i$ 
			and recalculate $\omega_{ji}^{\prime}=q_i^\TT Bs_j$.  
			\item[] Update
			$\omega_{jl}^{\prime} =\omega_{jl}^{\prime} -\tau\psi_{il}$ ($1\leq l\leq j$, $l\neq i$)  
			and $\varphi_{lj}^{\prime} =\varphi_{lj}^{\prime} -\tau\omega_{li}$ ($1\leq l<j$).
		\end{itemize}
		\item[] \textbf{end for}
		\item[\footnotesize (\ref{skwbld1}.4)] Update $\alpha_j=\|s_j\|_B$ 
		and compute $\varphi_{ij}=\frac{\varphi_{ij}^{\prime}}{\alpha_j}$ for $1\leq i<j$ and 
		$\omega_{ji}=\frac{\omega_{ji}^{\prime}}{\alpha_j}$ for $1\leq i\leq j$.  
	\end{itemize}
\end{algorithm}
  
\begin{algorithm}[tbph]
	\caption{Partial $B$-reorthogonalization in the GSSLBD process-Part II.}\label{alg4}
	\begin{itemize}[leftmargin = 2em]
		\item[\footnotesize (\ref{skwbld3}.1)] Compute
		$\psi_{1j}^{\prime},\psi_{2j}^{\prime},\dots,\psi_{jj}^{\prime}$ and
		$\omega_{1j}^{\prime},\omega_{2j}^{\prime},\dots,\omega_{jj}^{\prime}$ using
		\eqref{orthQQ}--\eqref{orthQP}, 
		and establish the index sets $\mathcal{I}_{Q,Q}^{(j)}$
		and $\mathcal{I}_{Q,P}^{(j)}$ based on \eqref{orthQQP}.
		
		\item[\footnotesize (\ref{skwbld3}.2)] \textbf{for}
		$i\in\mathcal{I}_{Q,Q}^{(j)}$ \textbf{do} \hspace{12em}
		\textcolor{gray}{$\%$ partially $B$-reorthogonalize
			$q_{j+1}$ to $Q_{j}$}
		\begin{itemize}[leftmargin=1.5em]
			\item[] Compute $\tau=q_i^\TT Bt_j$, overwrite $t_j\circeq t_j-\tau q_i$ 
			and recalculate $\psi_{i,j+1}^{\prime}=q_i^\TT Bt_j$. 
			\item[] Update
			$\psi_{l,j+1}^{\prime}\circeq\psi_{l,j+1}^{\prime}\!-\!\tau \psi_{li}$ 
			($1\leq l\leq j$, $l\neq i$) and 
			$\omega_{l,j+1}^{\prime}\circeq\omega_{l,j+1}^{\prime}
			-\tau\omega_{li}$ ($1\leq l\leq j$).
		\end{itemize}
		\item[] \textbf{end for}
		
		\item[\footnotesize (\ref{skwbld3}.3)] \textbf{for}
		$i\in\mathcal{I}_{Q,P}^{(j)}$ \textbf{do}
		\hspace{12em} \textcolor{gray}{$\%$ partially
			$B$-reorthogonalize $q_{j+1}$ to $P_{j}$ }
		\begin{itemize}[leftmargin=1.5em]
			\item[] Compute $\tau=p_i^\TT Bt_j$, overwrite $t_j=t_j-\tau p_i$ 
			and recalculate $\omega_{i,j+1}^{\prime}=p_i^\TT Bt_j$. 
			\item[] Update
			$\omega_{l,j+1}^{\prime}\circeq\omega_{l,j+1}^{\prime}-\tau 
			\varphi_{li}$ ($1\leq l\leq j$, $l\neq i$)  and 
			$\psi_{l,j+1}^{\prime}\circeq\psi_{l,j+1}^{\prime}-\tau \omega_{il}$ 
			($1\leq l\leq j$).
		\end{itemize}
		\item[] \textbf{end for}
		\item[\footnotesize (\ref{skwbld3}.4)] Update $\beta_j=\|t_j\|_B$ and calculate 
		$\psi_{i,j+1}=\frac{\psi_{i,j+1}^{\prime}}{\beta_j}$ 
		and $\omega_{i,j+1}=\frac{\omega_{i,j+1}^{\prime}}{\beta_j}$ 
		for $1\leq i\leq j$.		
	\end{itemize}
\end{algorithm} 

After those, the GSSLBD process proceeds and computes the right generalized 
Lanczos vector $p_{j+1}$.  
According to \eqref{BorthQQ}--\eqref{BorthQP}, we calculate the scaled new 
quantities in $\Psi$ and $\Omega$ by    
\begin{eqnarray}
&&\psi_{i,j+1}^{\prime}=\beta_{i-1}\varphi_{i-1,j}
+\alpha_i\varphi_{ij}-\alpha_j\psi_{ij}, \qquad 
\psi_{i,j+1}^{\prime}\circeq\psi_{i,j+1}^{\prime}
+\sign(\psi_{i,j+1}^{\prime})\tilde\varepsilon, \label{orthQQ}\\
&&\omega_{i,j+1}^{\prime}=-\beta_{i}\omega_{j,i+1}
-\alpha_i\omega_{ji}-\alpha_{j}\omega_{ij},\qquad\hspace{0.05em}
\omega_{i,j+1}^{\prime}\circeq\omega_{i,j+1}^{\prime}
+\sign(\omega_{i,j+1}^{\prime})\tilde\varepsilon. \label{orthQP}
\end{eqnarray} 
Then we have $\psi_{i,j+1}=\psi_{i,j+1}^{\prime}/\beta_j$ and 
$\omega_{i,j+1}=\omega_{i,j+1}^{\prime}/\beta_j$ for $1\leq i\leq j$. 
Based on the sizes of them, we determine the index sets 
\begin{equation}\label{orthQQP}
  \mathcal{I}_{Q,Q}^{(j)}=\left\{1\leq i\leq j \big|
	  |\psi_{i,j+1}^{\prime}|\geq\beta_j\sqrt{\tfrac{\epsilon}{m}}\right\},\quad 
	\mathcal{I}_{Q,P}^{(j)}=\left\{1\leq i\leq j\big|
  |\omega_{i,j+1}^{\prime}|\geq\beta_j\sqrt{\tfrac{\epsilon}{m}}\right\},  
\end{equation}
which include the indices corresponding to the previous right and left 
generalized Lanczos vectors $q_i$ and $p_i$ that $q_{j+1}$ has lost 
semi-$B$-orthogonality and semi-$B$-biorthogonality to, respectively. 
We then adopt the generalized MGS procedure to $B$-reorthogonal 
$q_{j+1}$ first against those $q_i$ with $i\in\mathcal{I}_{Q,Q}^{(j)}$ and then 
against those $p_i$ with $i\in\mathcal{I}_{Q,P}^{(j)}$. 
Meanwhile, we update the quantities $\psi_{lj}$'s and $\omega_{jl}$'s for 
$l=1,\dots,j$ efficiently. 
We summarize these in Algorithm~\ref{alg4}, which is inserted between 
steps~\ref{skwbld3} and \ref{skwbld4} of Algorithm~\ref{alg2}. 
 
So far, we have evolved the $m$-step GSSLBD process with partial 
$B$-reorthogonalization.  
We remark that for $B$ ill conditioned, the applications of $B^{-1}$ 
at steps \ref{skwbld1} and \ref{skwbld3} of algorithm~\ref{alg2} 
should introduce much larger sized backward errors in the computed 
generalized Lanczos vectors than it does for a well conditioned $B$, 
even if the relative residual norms of the underlying linear 
systems with $B$ the coefficient matrix have attained the level 
of working precision, causing that the computed generalized Lanczos 
vectors loss the valid semi-$B$-orthogonality and semi-$B$-orthogonality 
faster. 
For the robustness of the GSSLBD process, one should implement the 
$B$-reorthogonalizations more frequently for an ill conditioned $B$. 
To this end, we have involve the $\kappa(B)$ term when 
introducing $\tilde\varepsilon$ in \eqref{varepsilon}. 
Extensive numerical experiments have confirmed the necessity and 
effectiveness of involving such an $\tilde\varepsilon$ in the 
partial $B$-orthogonalization strategy for the GSSLBD process, 
as we shall see later. 
Moreover, from Algorithms~\ref{alg3}--\ref{alg4} and expressions 
\eqref{orthPP}--\eqref{orthPQ} and \eqref{orthQQ}--\eqref{orthQP}, 
the $j$th step of GSSLBD process consumes $\mathcal{O}(j)$ extra flops 
to calculate the scaled new entries of $\Phi$, $\Psi$ and $\Omega$ 
for the first time, which is neglectable relative to the cost of 
forming matrix-vector multiplications with $A$ and $B^{-1}$, 
each for twice, at this step. 
Additionally, it costs $\mathcal{O}(j)$ extra flops to update them 
after one step of $B$-reorthogonalization, which is 
negligible compared to the $\mathcal{O}(n)$ flops used by 
the $B$-reorthogonalization process. 
In conclusion, we are able to estimate the levels of $B$-orthogonality
and $B$-biorthogonality among the generalized Lanczos vectors  efficiently. 

\subsection{Implicit restart}\label{subsec:4}
When the dimension of subspaces reaches the maximum number 
$m$ allowed, we perform an implicit restart in the GSSLBD process. 
To this end, we need to select $m-k$ shifts, for which we 
take as the undesired Ritz values  
$\mu_1=\theta_{k+1},\dots, \mu_{m-k}=\theta_{m}$ computed during 
\eqref{Ritz} if the largest conjugate eigenpairs in magnitude of 
$(A,B)$ are sought; and if the smallest $k$ ones are wanted, 
we take $\mu_1=\theta_{1},\dots, \mu_{m-k}=\theta_{m-k}$.  
Particularly, if a shift is too close to some of the desired 
generalized spectra values of $(A,B)$, we replace it with a 
more appropriate value. 
We adopt the strategies suggested in 
\cite{jia2003implicitly,larsen1998lanczos}. 
Specifically, for the computation of the largest eigenpairs 
of $(A,B)$, if a shift $\mu$ satisfies 
\begin{equation}\label{badshift1}
\left|(\theta_k-\|r_{\pm k}\|)-\mu\right|\leq{\theta_k}\cdot10^{-3}  
\end{equation}
with $r_{\pm k}$ the residual of the approximate eigenpairs  
$(\tilde\lambda_{\pm k},\tilde x_{\pm k})$ defined by 
\eqref{defres}, we regard it as a bad one and reset it as \emph{zero}; 
as for the computation of the smallest eigenpairs of $(A,B)$, 
we take $\mu$ as a bad shift and reset it as the \emph{largest Ritz value $\theta_1$}  
if it accomplishes
\begin{equation}\label{badshift2}
	\left|(\theta_{m-k+1}+\|r_{\pm (m-k+1)}\|)
	-\mu\right|\leq\theta_{m-k+1} \cdot 10^{-3}.
\end{equation}
 
With these selected and modified shifts, the implicit restart module 
implements a series of Givens rotations on the bidiagonal 
matrix $G_m$ defined in \eqref{LBDmat} from the left and 
right successively, and at the same time accumulates the 
involved orthogonal matrices, which is equivalent to 
performing $m-k$ implicit QR iterations 
\cite{saad2011numerical,stewart2001matrix} on the triangular 
matrix $G_m^\TT G_m$ with the shifts $\mu_1^2,\dots,\mu_{m-k}^2$.  
As a consequence, we obtain two $m$-by-$m$ orthogonal matrices 
$\widetilde C_m$ and $\widetilde D_m$, whose lower bandwidths 
are both $m-k$, and the bidiagonal projection matrix 
$\widetilde G_m =\widetilde C_m^\TT G_m\widetilde D_m$; 
for details, see \cite{bjorck1994implicit,jia2003implicitly,
	jia2010refined,larsen2001combining}. 
Therefore, we first compute the vector 
\begin{equation}\label{restart1}
	\tilde q_{k+1}=\beta_{m}\tilde c_{mk} q_{m+1}
	+ \tilde \beta_{k} Q_m\tilde d_{k+1}, 
\end{equation}
where $\tilde c_{mk}$ and $\tilde \beta_{k}$ are the $(m,k)$- 
and $(k,k+1)$-element of $\widetilde C_m$ and $\widetilde G_m$, 
respectively, and $\tilde d_{k+1}$ is the $(k+1)$th column of 
$\widetilde D_m$.  
Then we update 
\begin{equation}\label{restart2}
	G_{k}\circeq\widetilde G_k,\quad
	P_{k}\circeq P_m\widetilde C_k,\quad
	Q_{k}\circeq Q_m\widetilde D_k,	\quad
	\beta_k\circeq\|\tilde q_{k+1}\|_B,\quad 
	q_{k+1}\circeq  \tilde q_{k+1}/\beta_k,
\end{equation}
where $\widetilde G_k$ is the $k$th leading principal submatrix 
of $\widetilde G_m$, and $\widetilde C_k$ and $\widetilde D_k$ 
consist of the first $k$ columns of $\widetilde C_m$ and 
$\widetilde D_m$, respectively. 
As a result, we have obtained a $k$-dimensional GSSLBD process 
\eqref{LBDmat2} rather than computing it from scratch. 
The whole computation and updating cost $\mathcal{O}(kmn)$ flops in total. 
We then extend the $k$-dimensional GSSLBD process to $m$-step 
in the way described in Section~\ref{sec2}. 
Repeat this process until convergence occurs. 

Remarkably, during the implicit restart, we can efficiently 
update the estimate matrices $\Phi,\Psi$ and $\Omega$ using 
\eqref{restart1}--\eqref{restart2}. 
Specifically, denote by $\Phi^{\prime}\in\mathbb{R}^{k\times k}$, 
$\Psi^{\prime}\in\mathbb{R}^{(k+1)\times (k+1)}$ and 
$\Omega^{\prime}\in\mathbb{R}^{k\times (k+1)}$ 
the new estimate matrices. 
Then their $k\times k$ leading principle matrices, 
denoted by $(\cdot)_{k}$, are computed by
$$
\Phi^{\prime}_k=\widetilde C_k^\TT \Phi_m  \widetilde C_k,\qquad
\Psi_k^{\prime}=\widetilde D_k^\TT \Psi_{m} \widetilde D_k,\qquad
\Omega_k^{\prime}=\widetilde C_k^\TT \Omega_{m}\widetilde D_k,
$$ 
In addition, the $(k+1)$th columns of $\Psi^{\prime}$ 
and $\Omega^{\prime}$, denoted by $(\cdot)_{:,k+1}$, are 
\begin{eqnarray*} 
\Psi_{:,k+1}^{\prime}   &=& \widetilde D_k^\TT 
	(\beta_m\tilde c_{mk}\Psi_{:,m+1}
	+ \tilde \beta_{k}\Psi_{m}\tilde d_{k+1})/\beta_k, \\
\Omega_{:,k+1}^{\prime} &=& \widetilde C_k^\TT 
	(\beta_m\tilde c_{mk}\Omega_{:,m+1}
	+ \tilde \beta_{k}\Omega_{m}\tilde d_{k+1})/\beta_k.
\end{eqnarray*} 
Thus we have obtained all the new estimate matrices by noticing that 
$\Psi^{\prime}$ is symmetric. 
We then overwrite $\Phi\circeq\Phi^{\prime}$, 
$\Psi\circeq\Psi^{\prime}$ and 
$\Omega\circeq\Omega^{\prime}$. 
This updating module consumes $\OO(m^2k)$ flops in total, 
negligible compared to the cost used by the implicit 
restart as $m\ll\ell\leq n/2$.  
 
\subsection{Implicitly restart GSSLBD algorithm with partial 
	$B$-reorthogonalization}\label{subsec:5}
\newcommand{\ee}{{\mathrm{e}}}
As mentioned previously, it is important for the reliability of 
the convergence criterion \eqref{res} and the efficient partial 
$B$-reorthogonalization module (cf.~\eqref{orthPP}--\eqref{orthPQ} 
and \eqref{orthQQ}--\eqref{orthQP}) to compute estimates 
for $\|H\|$, $\|B\|$ and $\kappa(B)$, denoted as 
$\|H\|_\ee$, $\|B\|_\ee$ and $\kappa_\ee(B)$, respectively.  
Fortunately, the largest Ritz value $\theta_1$ computed during the 
GSSLBD method is quite a reasonable approximation to $\|H\|$. 
Therefore, we simply set $\|H\|_\ee = \theta_1$. 
As for $\|B\|$ and $\kappa(B)$, we can use the Matlab built-in 
functions {\sf normest} and  {\sf condest} to estimate the $2$-norm 
and $1$-norm condition number of $B$, and take the results as  
$\|B\|_\ee$ and $\kappa_\ee(B)$, respectively. 
Particularly, if $B$ is given as a function handle in which 
the matrix-vector multiplications with $B$ and $B^{-1}$ are specified, 
then we perform the $m_L$-step Lanczos process on 
$B$ and $B^{-1}$ with a reasonably small $m_L$ to compute approximations 
to their largest eigenvalues, denoted as $\lambda_{\max,\ee}(B)$ and 
$\lambda_{\max,\ee}(B^{-1})$,  
and then we take $\|B\|_e=\lambda_{\max,\ee}(B)$ and 
$\kappa_\ee(B)=\lambda_{\max,\ee}(B)\cdot\lambda_{\max,\ee}(B^{-1})$. 
To be efficient, we take an $m_L=30$.
 
\begin{algorithm}[tbph]
	\caption{Implicitly restarted GSSLBD algorithm with partial $B$-reorthogonalization}
	\renewcommand{\algorithmicrequire}{\textbf{Input: }}
	\renewcommand{\algorithmicensure} {\textbf{Output: }}	
	\begin{algorithmic}[1]\label{alg5}
		\STATE{Initialization: Set $i=0$, $P_0=[\ \ ]$, $Q_1=[q_1]$ and $B_0=[\ \ ]$.}
		\WHILE{$i<i_{max}$}
		
		\STATE {Perform the $m$- or $(m-k)$-step GSSLBD
			process with partial $B$-reorthogonalization.}
		
		\STATE{Compute the SVD \eqref{Ritz} of $G_m$ to obtain its singular 
			triplets $\left(\theta_j,c_j,d_j\right)$, $j=1,\dots,m$.  
			Calculate the residual norms $\|r_{\pm j}\|$ according to 
			\eqref{resieasy}, $j=1,\dots,m$.}
			
		\IF{\label{step5}$\|r_{\pm j}\|\leq \sqrt{\|B\|_\ee}\|H\|_\ee \cdot tol$, $\forall j\in \mathcal{J}$}
		\STATE{\label{step6}Form $\tilde u_j=P_mc_j$ and $\tilde v_j=Q_md_j$, and \textbf{break} the loop.}
		\ENDIF\label{step7}
		
\STATE{Set the implicit restarting shifts $\mu_1,\dots,\mu_{m-k}$ as described in section~\ref{subsec:4}.}
		
\STATE{ Perform the implicit restart with the selected shifts. Set $i=i+1$ 
and goto step 3.}

		\ENDWHILE
	\end{algorithmic}
\end{algorithm}

In Algorithm~\ref{alg5} we sketch the implicitly restarted GSSLBD algorithm with 
partial $B$-reorthogonalization for computing the $2k$ extreme conjugate eigenpairs 
of $(A,B)$, where in steps~\ref{step5} the index set 
$\mathcal{J}=\{1,\dots,k\}$ if the $2k$ largest eigenvalues in magnitude 
are sought and $\mathcal{J}=\{m-k+1,\dots,m\}$ if the smallest $2k$ ones are wanted.  
It demands the devices to form matrix-vector multiplications with $A$, 
$B$ and $B^{-1}$, the number $2k$ of desired eigenpairs, 
and the stopping tolerance $\textit{tol}$ in \eqref{res}. 
Optional parameters include the maximum dimension $m$ of the subspaces, 
the maximum number $i_{\max}$ of total implicit restarts and the $B$-norm 
unit-length starting vector $q_1$, which, by default, are set as $30$, 
$2000$ and $e/\|e\|_B$ or $Ae/\|Ae\|_B$ with $e=[1,\dots,1]$  
for computing the largest and smallest eigenpairs, respectively.

\section{Numerical examples}\label{sec:5}
We report numerical experiments on several matrix pairs to 
illustrate the potential and performance of our implicitly 
restarted GSSLBD algorithm, abbreviate as IRGSSLBD, which 
has been written and coded in the MATLAB language, for the
extreme generalized eigenproblems of large skew-symmetric/symmetric 
positive definite matrix pairs. 
All the experiments were implemented on an Intel (R)
core (TM) i9-10885H CPU 2.40 GHz with the main memory
64 GB and 16 cores using the Matlab R2022b with the
machine precision $\epsilon=2.22\times 10^{-16}$ under
the Microsoft Windows 10 64-bit system. 

We always run IRGSSLBD with the default parameters 
described in  Section~\ref{sec:4}. 
Also, we test the Matlab built-in function {\sf eigs} with 
the same parameters and make a comparison with IRGSSLBD 
to illustrate the overall efficiency of our algorithm.  
When computing the smallest conjugate eigenpairs of $(A,B)$, 
the shift-and-invert approach in {\sf eigs} requires  
the LU factorization of $A$, which may be very demanding 
in computational cost and storage, so that at each step 
a linear equation with $A$ the coefficient matrix can be 
solved directly. In this sense, there is no fair way to 
compare our IRGSSLBD with {\sf eigs}. 
Therefore, when computing the smallest eigenpairs, we only 
report the results of our IRGSSLBD so as to 
demonstrate that it is capable of accomplishing such tasks. 
We record the number of matrix-vector multiplications with $A$ 
or $B^{-1}$, abbreviated as $\#{\rm Mv}$, and the CPU time in 
second, denoted by $T_{\rm cpu}$,
that the two algorithms use to achieve the same stopping tolerance. 
We should remind that a comparison of CPU time used by IRGSSLBD and
{\sf eigs} may be misleading because our IRGSSLBD algorithm 
is coded in pure Matlab language while {\sf eigs} is programmed 
with the advanced and much higher efficient C/C++ language.
Nevertheless, we aim to illustrate that IRGSSLBD is
really fast and competitive with {\sf eigs} even in terms of CPU time.

\begin{exper}\label{exper1}
We compute $20$ largest in magnitude eigenvalues and the corresponding 
eigenvectors of $(A,B)=(\mathrm{plsk1919},T_n(\rho,\delta))$ , 
where $plsk1919$ is a $1919$-by-$1919$ realistic real skew-symmetric 
matrix from the University of Florida Sparse Matrix Collection
\cite{davis2011university}, and the symmetric positive definite Toeplitz matrix
\begin{equation}\label{defTn}
T_n(\rho,\delta) = \begin{bmatrix}\rho&\delta&&\\
	\delta&\ddots&\ddots&\\&\ddots&\ddots&\delta
	\\&&\delta&\rho\end{bmatrix}\in\mathbb{R}^{n\times n}	
\end{equation} 
with $\rho>2\delta$ is some discrete smoothing norm matrix that is 
widely used in linear discrete ill-posed problems with general-form 
of the Tikhonov regularization; see, e.g., \cite[pp.12--3]{hansen1998rank}. 
Taking $\delta=1$ and $\rho = 3$ and $2.000001$ respectively, we obtain 
two problems.  
\end{exper}

\begin{table}[tbhp]
	\caption{Results on $(A,B)=(\mathrm{plsk1919},T(\rho,\delta))$ with $\delta=1$.}\label{table1}
	\begin{center}
		\begin{tabular}{ccccccc} \toprule
			\multirow{2}{*}{\ \ $\rho$ \ \ } 
			&\multirow{2}{*}{\ \ $\kappa(B)$ \ \ }
			&\multirow{2}{*}{\ \ $\gap(10)$ \ \ } 
			&\multicolumn{2}{c}{{\sf eigs}}
			&\multicolumn{2}{c}{IRGSSLBD}  \\
			  \cmidrule(lr){4-5} \cmidrule(lr){6-7}
			& & 
			&\ \  $\#$Mv\ \  &\ \ \ $T_{\mathrm{cpu}}$\ \ \ \ \
			&\ \  $\#$Mv\ \  &\ \ \ $T_{\mathrm{cpu}}$\ \ \ \ \ \\ \midrule
			$3$ &5.00 &9.42e-3 &152 &2.13e-2 &110 &2.94e-2 \\
			$2.000001$ 	&1.09e+6 &5.60e-2 &50 &8.07e-3 &60 &1.48e-2 \\ 
			\bottomrule
		\end{tabular}
	\end{center}
\end{table}

Table~\ref{table1} displays some basic properties of these two problems 
as well as the $\#$MV and CPU time used by IRGSSLBD and {\sf eigs} to 
achieve the same convergence, where, by Theorem~\ref{thm2}, the parameter
\begin{equation*} 
\gap(k) := \min_{1\leq j\leq k}\frac{|\lambda_j^2-\lambda_{j+1}^2|}{|\lambda_{j+1}^2-\lambda_{\ell}^2|}
=\min_{1\leq j\leq k}\frac{\sigma_j^2-\sigma_{j+1}^2}{\sigma_{j+1}^2-\sigma_{\ell}^2}	
\end{equation*}
indicates that IRGSSLBD can converge pretty fast for both problems, 
as confirmed by the followed computational results. 
For the first matrix pair $(A,B)$, $B$ is well conditioned. 
We observe that both algorithms succeed to compute the desired 
eigenvalues of $(A,B)$ with high accuracy, i.e., the relative 
errors at the level of $\mathcal{O}(10^{-14})$, and the corresponding 
eigenvectors with the errors at the level of $\mathcal{O}(10^{-8})$. 
Between them, IRGSSLBD outperforms {\sf eigs} by using $27.63\%$ 
fewer $\#$MV, i.e., iterations, and similar CPU time. 
For the second problem, $B$ is ill conditioned and both algorithms 
compute the desired eigenvalues with lower accuracy, i.e., the relative 
errors at the level of $\mathcal{O}(10^{-12})$, and the associated 
eigenvectors still at the level of $\mathcal{O}(10^{-8})$. 
IRGSSLBD uses comparable CPU time to {\sf eigs} even thought 
consuming slightly more $\#$MV than the latter. 
In this sense, IRGSSLBD is comparable to {\sf eigs} for this problem. 

Particularly, we have also tested these two problems using the IRGSSLBD 
algorithm with full $B$-reorthogonalization, where each computed 
generalized Lanczos vector has been $B$-reorthogonalized to all the 
previous ones so that all those vectors are kept $B$-reorthogonal 
to the level of working precision. 
We observe that for both problems, IRGSSLBD with partial 
and full $B$-reorthogonalization can compute all the desired 
eigenvalues equally accurately using the same $\#$MV. 
However, with partial $B$-reorthogonalization, IRGSSLBD 
implements only $758$ and $1194$ $B$-reorthogonalizations 
in total for the first and second problems, respectively, 
significantly fewer than the $3861$ and $1830$ ones used by 
IRGSSLBD with full $B$-reorthogonalization. 
This illustrates the effectiveness and efficiency of the 
partial $B$-reorthogonalized strategy designed for the GSSLBD process.  
Moreover, just as mentioned at the end of section~\ref{subsec:3}, 
we observe that the partial $B$-reorthogonalization strategy proposed 
indeed involves relatively more frequent $B$-reorthogonalizations 
for ill conditioned $B$. 

\begin{exper}\label{exper2}
We compute $20$ absolutely largest eigenvalues and the corresponding 
eigenvectors of the $n$-by-$n$ matrix pair $(A,B)$ with  
\begin{eqnarray*} 
	A&=&I_{j}\otimes I_{j}\otimes S_{j}(\upsilon_1)+
	I_{j} \otimes S_{j}(\upsilon_2)\otimes I_{j}+
	S_{j}(\upsilon_3)\otimes I_{j}\otimes I_{j},  \\
	B&=&I_{j}\otimes I_{j}\otimes T_{j}(\rho,\delta)	+
	I_{j} \otimes T_{j}(\rho,\delta)	\otimes I_{j}+
	T_{j}(\rho,\delta)	\otimes I_{j}\otimes I_{j}, 
\end{eqnarray*}
where $n=j^3$, $T_j(\rho,\delta)$ is defined by \eqref{defTn} and 
\begin{equation}\label{defSj}
	S_j(\upsilon)=\begin{bmatrix}
		0&\upsilon&&\\
		-\upsilon&\ddots&\ddots&\\
		&\ddots&\ddots&\upsilon\\&&-\upsilon&0
	\end{bmatrix} \in\mathbb{R}^{j\times j}. 	
\end{equation}
Such $A$ comes from a realistic problem of \cite[section 11]{greif2016} 
that can be explained as the approximate finite difference 
discretization of a constant convective term on the unit cube 
$(0, 1) \times  (0, 1) \times  (0, 1)$ in three dimensions, 
and $B$ is the corresponding discrete smoothing norm matrix.
Taking $\upsilon_1 =0.4, \upsilon_2 =0.5, \upsilon_3 =0.6$, 
$j = 32$ and $64$, respectively,  and $\delta=1$, $\rho=3$ 
and $2.000001$, respectively, we obtain two $32768$-by-$32768$ 
generalized eigenproblems and two $262144$-by-$262144$ problems.  
\end{exper} 
 
\begin{table}[tbhp]
	\caption{Results of Experiment~\ref{exper2}.}\label{table2}
	\begin{center}
		\begin{tabular}{cccccccc} \toprule
			\multirow{2}{*}{$\rho$} &\multirow{2}{*}{$n$} 
			&\multirow{2}{*}{\ \ $\kappa(B)$\ \ }
			&\multirow{2}{*}{\ \ $\gap(10)$\ \ } 
			&\multicolumn{2}{c}{{\sf eigs}}
			&\multicolumn{2}{c}{IRGSSLBD}  \\
			\cmidrule(lr){5-6} \cmidrule(lr){7-8}
			&& & 
			&\ \  $\#$Mv\ \  &\ \ \ $T_{\mathrm{cpu}}$\ \ \ \ \
			&\ \  $\#$Mv\ \  &\ \ \ $T_{\mathrm{cpu}}$\ \ \ \ \ \\ \midrule 
			\multirow{2}{*}{$3$} &$32768$ &4.95  &6.39e-4  &880  &12.8 &386  &6.60    \\ 
			 &$262144$ 	&4.99  &1.57e-4  &2770  &7.82e+2 &744  &2.46e+2   \\ \midrule  
			\multirow{2}{*}{$2.000001$} &$32768$  &4.41e+2  &1.60e-3  &102  &1.73 &94  &2.31   \\ 
		 &$262144$ 	&1.71e+3  &4.10e-4  &110  &39.9 &84  &81.3   \\
			\bottomrule
		\end{tabular}
	\end{center}
\end{table}
 
Table~\ref{table2} reports the basic properties of these four problems 
along with the computational results of the two algorithms. 
As we can see, for $\rho=3$, the $B$'s are well conditioned and the 
desired eigenvalues of $(A,B)$ are highly clustered each other. 
Therefore, the two algorithms uses numerous iterations to converge. 
Specifically, for $n = 32768$, IRGSSLBD reduces $56.14\%$ $\#$Mv 
and $48.44\%$ CPU time from {\sf eigs}. 
The advantage is more obvious for $n=262144$ as the former algorithm 
only consumed $26.86\%$ of the $\#$Mv and $31.46\%$ of the CPU time 
used by the latter one to achieve the same convergence. 
Clearly, IRGSSLBD outmatches {\sf eigs} substaintially for these two problems. 

For $\rho=2.000001$, the $B$'s are ill conditioned but the desired 
eigenvalues of $(A,B)$ are slightly better separated than those of 
the first two problems and only some of them are clustered. 
Consequently, both algorithms converge fast. 
Compared to {\sf eigs}, IRGSSLBD reduces $7.84\%$ and $23.64\%$ $\#$Mv 
for $n=32768$ and $262144$, respectively.
We remark that $A$ of these two problems are 
highly sparse and even sparser than $B$.
The gain of IRGSSLBD in fewer $\#$MV is compromised by the extra 
cost of performing $B$-reorthogonalizations. 
As a consequence, it uses $33.53\%$ and $103.8\%$ more CPU time than {\sf eigs}, respectively. 
Nevertheless, for these two problems IRGSSLBD is slightly superior 
to {\sf eigs} regarding total iterations. 
 
\begin{exper}\label{exper3}
	For the matrix split-based iterative linear solver for $Cx=b$ 
	where $C$ is split as the sum of its skew-symmetric part 
	$A=\frac{C-C^T}{2}$ and symmetric part $B=\frac{C+C^T}{2}$, 
	the updating formula for the approximate solution is 
	$x_{j+1}=-B^{-1}Ax_{j}+B^{-1}b$. 
	It is important to determine the spectral radius of $B^{-1}A$, 
	i.e., that of $(A,B)$, for predicting the convergence behavior 
	of the method.  
	We select three matrices $C=\mathrm{cage10}$, $\mathrm{memplus}$ 
	and $\mathrm{viscorocks}$ from \cite{davis2011university}  
	and compute $20$ largest eigenpairs in magnitude of the split 
	matrix pairs $(A,B)=\left(\frac{C-C^T}{2},\frac{C+C^T}{2}\right)$. 
\end{exper} 

\begin{table}[tbhp]
	\caption{Results of Experiment~\ref{exper3}, where $nnz$ denotes 
		the total number of nonzero elements of $A$ and $B$.}\label{table3}
	\begin{center}
		\begin{tabular}{ccccccccc} \toprule
			\multirow{2}{*}{$C$} &\multirow{2}{*}{$n$} &\multirow{2}{*}{$nnz$}
			&\multirow{2}{*}{$\kappa(B)$}
			&\multirow{2}{*}{$\gap(10)$} 
			&\multicolumn{2}{c}{{\sf eigs}}
			&\multicolumn{2}{c}{IRGSSLBD}  \\
			\cmidrule(lr){6-7} \cmidrule(lr){8-9}
			&&& & 
			&$\#$Mv&$T_{\mathrm{cpu}}$
			&$\#$Mv&$T_{\mathrm{cpu}}$\\ \midrule 
			$\mathrm{cage10}$ &11397&265621&12.0&1.18e-2&170&2.60&124&2.81 \\
			$\mathrm{memplus}$ &17758&143508&1.29e+5 &2.40e-7&71058\!&42.6&230&0.70 \\
			$\mathrm{viscorocks}$&37762&1923224&9.77e+11 &3.06e-2&122&0.95&60&0.96 \\		 
			\bottomrule
		\end{tabular}
	\end{center}
\end{table}
 
Table~\ref{table3} displays the basic properties of these 
problems and the computational results of IRGSSLBD and {\sf eigs}. 
We observe very similar phenomenon to the previous experiments. 
As can be seen, for $C\!=\!\mathrm{cage10}$ 
 and $\mathrm{viscorocks}$, the desired 
eigenvalues of $(A,B)$ are quite well separated from each 
other and both algorithms converge fast. 
IRGSSLBD surpasses {\sf eigs} slightly for $C\!=\!\mathrm{cage10}$ 
and considerably for $C\!=\!\mathrm{viscorocks}$ in the sense of 
using $27.06\%$ and $50.82\%$ fewer $\#$MV and competitive CPU 
time compared to the latter, respectively.  
For $C\!=\!\mathrm{memplus}$, the desired eigenvalues of $(A,B)$ 
are extremely clustered each other and {\sf eigs} consumes numerous 
$\#$MV and long CPU time trying to compute the associated eigenpairs 
yet fails and wrongly outputs eight eigenpairs with zero eigenvalues. 
On the contrary, IRGSSLBD converges  fast and correctly and depletes only 
$3.24\mbox{\textperthousand}$ $\#$MV and $1.64\%$ CPU time of {\sf eigs}. 
Evidently, IRGSSLBD vanquishes {\sf eigs} significantly for this problem.

\begin{exper}\label{exper4}
	We compute $10$ eigenpairs associated with the $10$ absolutely smallest 
	eigenvalues of $(A,B)=(S_n(1),T_n(3,1))$   
	with $n=10000$ and $T_n$ and $S_n$  defined by \eqref{defTn}--\eqref{defSj}.     
\end{exper} 

For this problem, we have $\kappa(B)=5.00$ and $\gap(5)=1.10e-7$, 
where $\gap(k)$ for the $2k$ absolutely smallest eigenvalues of 
$(A,B)$ is defined by 
\begin{equation}\label{gapks}
	\gap(k) := \min_{1\leq j\leq k}\frac{|\lambda_{\ell-j+1}^2-\lambda_{\ell-j}^2|}{|\lambda_{\ell-j}^2-\lambda_{1}^2|}=
	\min_{1\leq j\leq k}\frac{\sigma_{\ell-j}^2-\sigma_{\ell-j+1}^2}{\sigma_{1}^2-\sigma_{\ell-j}^2},
\end{equation} 
the size of which, as shown by Theorem~\ref{thm2}, critically 
affects the performance of IRGSSLBD for computing the corresponding 
$2k$ eigenpairs; the larger $\gap(k)$ is, generally the faster 
IRGSSLBD converges. 
Such an extremely small $\gap(5)$ indicates that the task of 
this problem is extraordinarily hard to accomplish.   
Nonetheless, we observe that IRGSSLBD manages to compute 
all the desired eigenpairs of $(A,B)$, though at the cost of 
numerous, i.e., $310046$, $\#$MV and long CPU time, i.e., $169$ seconds. 

In summary, by selecting appropriate starting vectors, IRGSSLBD 
suits well for computing several extremal eigenpairs of a 
skew-symmetric/symmetric positive definite matrix pair. 
And it is more robust than {\sf eigs} with generally higher 
and even much higher overall efficiency.  

\section{Conclusions}\label{sec:6}
For a skew-symmetric/symmetric positive definite matrix pair $(A,B)$, 
we have proposed the GSSLBD process that produces two sets of $B$-orthonormal 
and $B$-biorthogonal generalized Lanczos vectors and a sequence of bidiagonal 
projection matrices of $A$ onto the generated Krylov subspaces with respect to $B$-norm. 
Based on this, we have developed a Rayleigh--Ritz 
projection-based GSSLBD method to compute several extreme conjugate  
eigenpairs of $(A,B)$ in  real arithmetic. 
We have made a rigorous theoretical analysis on the GSSLBD process and 
presented convergence results on how fast an eigenspace associated with 
a concerned conjugate eigenvalue pair of $(A,B)$ approaches the $B$-direct 
sum of the two Krylov subspaces. 
In terms of the $B$-distance between those two subspaces, we have 
established a priori error bounds for the computed conjugate 
approximate eigenvalues and the corresponding approximate eigenspaces. 
The results demonstrate that the GSSLBD method favors extreme 
eigenvalues and the corresponding eigenspaces. 

In finite precision arithmetic, we have proved that the 
semi-$B$-orthogonality and semi-$B$-biorthogonality of the left 
and right generalized Lanczos vectors produced by the GSSLBD 
process suffice to compute the eigenvalues of $(A,B)$ accurately. 
We designed an efficient and valid strategy 
to track the levels of $B$-orthogonality and $B$-biorthogonality 
between these vectors and, based on that, proposed an 
effective partial $B$-reorthogonalization scheme for the GSSLBD 
process to keep the desired semi-$B$-orthogonality and semi-$B$-biorthogonality.
To be practical, we have adapt the implicit restart technique to 
the GSSLBD process with partial $B$-reorthogonalization and developed 
an implicit restart GSSLBD algorithm to compute several extreme conjugate 
eigenvalues and the associated eigenvectors 
of a large skew-symmetric/symmetric positive definite matrix pair. 
Numerical experiments have illustrated the robustness and 
efficiency of the implicit restart GSSLBD algorithm. 

\section*{Declarations}

The author declares that she has no financial interests, 
and she read and approved the final manuscript.
The algorithmic Matlab code is available upon reasonable 
request from the author.

\bibliographystyle{siamplain}
\bibliography{gsslbdep}
\end{document}